\documentclass[reqno,10pt]{amsart}

\usepackage{amssymb,latexsym}
\usepackage{mltex}
\usepackage{amsmath}
\usepackage[all]{xy}
 \usepackage{lscape}
\usepackage{enumerate}
\usepackage{amsthm}
\usepackage{hyperref}

\newtheorem{thm}{Theorem}
\newtheorem{cor}{Corollary}
\newtheorem{lem}{Lemma}[section]
\newtheorem{prop}[lem]{Proposition}
\newtheorem{rem}[lem]{Remark}

\newcommand{\lgra}{\longrightarrow}

\newcommand{\trace}{\mathrm{tr\,}}

\newcommand{\C}{\mathbb{C}}

\newcommand{\HH}{\mathbb{H}}
\newcommand{\HA}{\mathbb{H}^{\mathcal{A}}}
\newcommand{\tch}{\hat}
\newcommand{\bigtch}{\widehat}
\newcommand{\unitquat}{\textit{1}}

\newcommand{\R}{\mathbb{R}}
\newcommand{\Q}{\mathbb{Q}}

\newcommand{\Z}{\mathbb{Z}}

\newcommand{\beqt}{\begin{equation}}  \newcommand{\eeqt}{\end{equation}}
\newcommand{\bal}{\begin{align}}      \newcommand{\eal}{\end{align}}
\newcommand{\ba}{\begin{array}}      \newcommand{\ea}{\end{array}}
\newcommand{\bc}{\begin{center}}     \newcommand{\ec}{\end{center}}
\newcommand{\be}{\begin{enumerate}}  \newcommand{\ee}{\end{enumerate}}
\newcommand{\beq}{\begin{eqnarray}}  \newcommand{\eeq}{\end{eqnarray}}
\newcommand{\beQ}{\begin{eqnarray*}} \newcommand{\eeQ}{\end{eqnarray*}}
\newcommand{\bi}{\begin{itemize}}    \newcommand{\ei}{\end{itemize}}
\newcommand{\bt}{\begin{tabular}}    \newcommand{\et}{\end{tabular}}
\newcommand{\finpreuve}{\hfill\square\\}

\begin{document}
\title{A spinor description of flat surfaces in $\R^4$}
\author{Pierre Bayard}
\thanks{The author was supported by the project CIC-UMSNH 4.5.}
\address{Pierre Bayard: Instituto de F\'{\i}sica y Matem\'aticas. U.M.S.N.H. Ciudad Universitaria. CP. 58040 Morelia, Michoac\'an, Mexico.}
\email{bayard@ifm.umich.mx}

\begin{abstract}
We describe the flat surfaces with flat normal bundle and regular Gauss map immersed in $\R^4$ using spinors and Lorentz numbers. We obtain a new proof of the local structure of these surfaces. We also study the flat tori in the sphere $S^3$ and obtain a new representation formula. We then deduce new proofs of their global structure, and of the global structure of their Gauss map image.
\end{abstract}
\maketitle
\markboth{PIERRE BAYARD}{A spinor description of flat surfaces in $\R^4$}

\section*{Introduction}

In this paper we are interested in flat surfaces with flat normal bundle in $\R^4$ and in their description using spinors. We show how a general spinor representation formula of surfaces in $\R^4$ permits to obtain the following important results concerning their local and global structure: 
\begin{enumerate}
\item locally a flat surface with flat normal bundle and regular Gauss map depends on four functions of one variable \cite{DT, CD};
\item a flat torus immersed in $S^3$ is a product of two horizontal closed curves in $S^3\subset \HH$ \cite{B,Sa,S,K}. 
\end{enumerate}
With this formalism we also obtain the structure of the Gauss map image of a flat torus in $S^3$ \cite{E,W}.
\\

This approach permits a unified treatment of numerous questions concerning flat surfaces with flat normal bundle in $\R^4:$ the main idea is to write the general spinor representation formula obtained in \cite{BLR} in parallel frames adapted to the surfaces; with the spinor construction and the representation formula at hand, the proofs then appear to be quite simple. 
\\

In the paper, the principal object attached to a flat immersion with flat normal bundle in $\R^4$ is a map $g,$ which represents a constant spinor field in a moving frame adapted to the immersion; as a consequence of the spinor representation formula, it appears that $g,$ together with the metric of the surface, determines the immersion (with an explicit formula). Moreover, in the special case of a surface in $S^3,$ $g$ also determines the metric of the surface and thus entirely determines the immersion. The structure of the flat immersions with flat normal bundle in $\R^4$ thus crucially relies on the structure of $g.$  In the paper, the map $g$ is considered as a curve into a sphere, parameterized by the Lorentz numbers. We show that its arc length parametrization yields natural coordinates on the surface which generalize the asymptotic Tchebychef net of the flat surfaces in $S^3$ (Theorem \ref{thm chart A} and Remark \ref{asymptotic net}); as a corollary of the spinor representation formula, we then obtain the local structure of the flat immersions with flat normal bundle in $\R^4$  (Theorem \ref{thm local description} and Corollary \ref{cor local description}). When the surface is compact, we show that the arc length parametrization of $g$ gives in fact a global parametrization of the surface (Proposition \ref{prop unif}), and, as a corollary, we obtain a new representation of the flat tori in $S^3$ (Theorem \ref{th formula F g}); we then observe that the two natural projections of $g$ on its positive and negative parts are in fact the two curves in the Kitagawa representation of the torus defined by $g$ (Theorem \ref{thm bianchi} and Section \ref{section Kitagawa representation}).  Finally, a Hopf projection of $g$ gives the Gauss map of the surface, which permits to study easily the structure of the Gauss map image of the flat tori in $S^3$ (Corollary \ref{corollary gauss map}).
\\

We previously used this approach in \cite{Bay} to study flat surfaces with flat normal bundle and regular Gauss map in 4-dimensional Minkowski space $\R^{1,3}.$ The present paper uses similar ideas, and gives applications of the spinor representation formula contained in \cite{BLR}.
\\

We quote papers concerning flat surfaces with flat normal bundle in $\R^4.$ In \cite{B} Bianchi constructed flat surfaces in $S^3$ as a product, in the quaternions, of two special curves. Sasaki \cite{Sa} and Spivak \cite{S} classified the complete flat surfaces in $S^3$. In \cite{K}, Kitagawa gave a method to construct all the flat tori in $S^3:$ a flat torus is a product of two curves in the unit quaternions; these curves are constructed as asymptotic lifts of periodic curves in $S^2$ (the Kitagawa representation). In \cite{GM}, G\'alvez and Mira constructed non-trivial flat tori with flat normal bundle in $\R^4$ which are not contained in any 3-sphere, which raises the question of the construction of all the flat tori with flat normal bundle in $\R^4.$ Enomoto \cite{E} and Weiner \cite{W} described the flat tori in $S^3$ in terms of their Gauss map.
\\

The outline of the paper is as follows: in Section \ref{section clifford algebra} we describe the Clifford algebra of $\R^4$ and its spinor representation using quaternions and Lorentz numbers, in Section \ref{section twisted spinor bundle} we describe the spinor bundle twisted by a bundle of rank two on a Riemannian surface and in Section \ref{section spinor representation} the spinor representation of a surface in  $\R^4,$ rewriting the spinor representation formula of \cite{BLR} using Lorentz numbers. Using the same formalism, we then describe in Section \ref{section gauss map} the Gauss map of a surface in $\R^4,$ together with its relation to the spinor representation formula of the surface. Section \ref{section lorentz surfaces} is devoted to Lorentz surfaces and Lorentz numbers. We then give the local description of the flat surfaces with flat normal bundle and regular Gauss map in $\R^4$ in Section \ref{section local description}, and finally describe the flat tori in $S^3$ and their Gauss map in Section \ref{section flat tori S3}. An appendix on Lorentz numbers and quaternions ends the paper. 
\section{Clifford algebra of $\R^4$ and spin representations with Lorentz numbers}\label{section clifford algebra}
In this section, we describe the Clifford algebras and spinors of $\R^4$ using Lorentz numbers and quaternions.
\subsection{Lorentz numbers and quaternions}\label{subsection quaternions}

The algebra of Lorentz numbers is the algebra
$$\mathcal{A}=\{x+\sigma y,\ x,y\in\mathbb{R}\},$$
with the usual operations, where $\sigma$ is a formal element satisfying $\sigma^2=1.$ We consider the quaternions with coefficients in $\mathcal{A},$
$$\HA:=\{a_0\unitquat+a_1 I+a_2J+a_3K,\ a_0,a_1,a_2,a_3\in\mathcal{A}\},$$
where $I,J,K$ satisfy
$$I^2=J^2=K^2=-\unitquat,\hspace{1cm} IJ=-JI=K.$$
For $\xi=a_0\unitquat+a_1I+a_2J+a_3K$ belonging to $\HA,$ we denote
$$\overline{\xi}:=a_0\unitquat-a_1I-a_2J-a_3K,$$
and we consider the following two inner products on $\HA:$
$$\langle\langle.,.\rangle\rangle:\hspace{.5cm}\begin{array}[t]{rcl}\HA\times\HA&\rightarrow&\HA\\
(\xi,\xi')&\mapsto&\overline{\xi'}{\xi}\end{array}$$
and
$$H:\hspace{.5cm}\begin{array}[t]{rcl}\HA\times\HA&\rightarrow&\mathcal{A}\\
(\xi,\xi')&\mapsto &a_0a_0'+a_1a_1'+a_2a_2'+a_3a_3'\end{array}$$
where $\xi=a_0\unitquat+a_1I+a_2J+a_3K$ and $\xi'=a_0'\unitquat+a_1'I+a_2'J+a_3'K.$ We note that
$$H(\xi,\xi)=\langle\langle\xi,\xi\rangle\rangle$$
for all $\xi\in\HA.$

\subsection{Clifford map, Clifford algebra and Spin representations}\label{section clifford map}

We consider the Clifford map
\begin{eqnarray*}
\R^{4}&\rightarrow&\HA(2)\\
(x_0,x_1,x_2,x_3)&\mapsto& \left(\begin{array}{cc}0&\sigma x_0\unitquat+x_1I+x_2J+x_3K\\-\sigma x_0\unitquat+x_1I+x_2J+x_3K&0\end{array}\right)
\end{eqnarray*}
where $\HA(2)$ stands for the space of $2\times 2$ matrices with coefficients in $\HA;$ using this map, the Clifford algebra of $\R^{4}$ identifies to
$$Cl(4)=\left\{\left(\begin{array}{cc}a&b\\\tch{b}&\tch{a}\end{array}\right),\ a,b\in\HA\right\}.$$
Here and below,  for $\xi=a_0\unitquat+a_1I+a_2J+a_3K\in\HA$ we denote
$$\tch{\xi}:=\overline{a_0}\unitquat+\overline{a_1}I+\overline{a_2}J+\overline{a_3}K,$$
where, if $a_i=x_i+\sigma y_i$ belongs to $\mathcal{A},$ its conjugate is $\overline{a_i}:=x_i-\sigma y_i.$ Its subalgebra of even elements is
\begin{equation}\label{even elements}
Cl_0(4)=\left\{\left(\begin{array}{cc}a&0\\0&\tch{a}\end{array}\right),\ a\in\HA\right\}\simeq\HA,
\end{equation}
and we have
\begin{equation}\label{spin4 S3A}
Spin(4)=\{a\in \HA:\ H(a,a)=1\}\hspace{.5cm}=:\ S^3_{\mathcal{A}}.
\end{equation}
We note the identification
\begin{equation}\label{S3A product}
S^3_{\mathcal{A}}\ \simeq\ S^3\times S^3
\end{equation}
when we identify $\HA$ to $\HH\oplus\HH$ by the isomorphism
\begin{eqnarray}\label{isom HA HH}
\HA&\rightarrow&\HH\oplus\HH\\
\xi&\mapsto& (\xi_+,\xi_-),\nonumber
\end{eqnarray} 
where $\xi,\xi_+$ and $\xi_-$ are linked by
$$\xi\ =\ \frac{1+\sigma}{2}\ \xi_+\ +\ \frac{1-\sigma}{2}\ \xi_-.$$
Moreover, we have the double covering
\begin{eqnarray}
\Phi:Spin(4)&\stackrel{2:1}{\longrightarrow}& SO(4)\label{double cover}\\
q&\mapsto &(\xi\in\R^{4}\mapsto q\xi \tch{q}^{-1}\in\R^{4}),\nonumber
\end{eqnarray}
where $\mathbb{R}^4\subset\HA$ is defined by 
\begin{eqnarray}\label{R4intoHA}
\mathbb{R}^4&:=&\{\xi\in \HA:\ \tch{\overline{\xi}}=-\xi\}\\
&=&\{\sigma x_0\unitquat+x_1I+x_2J+x_3K,\ (x_0,x_1,x_2,x_3)\in\mathbb{R}^4\}.\nonumber
\end{eqnarray}
We note that the euclidean metric on $\mathbb{R}^4$ is given by the restriction of the form $H.$ 
\\

Let $\rho:Cl(4)\rightarrow End_{\C}(\HA)$ be the complex representation of $Cl(4)$ on $\HA$ given by
\begin{eqnarray*}
\rho\left(\begin{array}{cc}a&b\\\tch{b}&\tch{a}\end{array}\right):&&\xi\simeq\left(\begin{array}{c}\xi\\\tch{\xi}\end{array}\right)\mapsto\left(\begin{array}{cc}a&b\\\tch{b}&\tch{a}\end{array}\right)\left(\begin{array}{c}\xi\\\tch{\xi}\end{array}\right)\simeq a\xi+b\tch{\xi},
\end{eqnarray*}
where the complex structure on $\HA$ is given by the right multiplication by $I.$ The restriction of the representation $\rho$ to $Spin(4)$ gives
\begin{eqnarray*}
\rho_{|Spin(4)}:\ Spin(4)&\rightarrow &End_{\C}(\HA)\\
a&\mapsto &(\xi\in\HA\mapsto a\xi\in\HA).
\end{eqnarray*}
This representation splits into 
$$\HA=\HA_+\oplus\HA_-$$
where $\HA_+=\{\xi\in\HA:\ \sigma\xi =\xi\}$ and $\HA_-=\{\xi\in\HA:\ \sigma\xi=-\xi\};$ this decomposition corresponds to (\ref{isom HA HH}), since 
$$\HA_+=\left\{\frac{1+\sigma}{2}\ \xi,\ \xi\in\HH\right\}\simeq\HH\hspace{.5cm}\mbox{and}\hspace{.5cm}\HA_-=\left\{\frac{1-\sigma}{2}\ \xi,\ \xi\in\HH\right\}\simeq\HH.$$

\subsection{Spinors under the splitting $\R^4=\R^2\times\R^2$}\label{section splitting}
We now consider the splitting $\R^4=\R^2\times\R^2$ and the corresponding inclusion 
$$SO(2)\times SO(2)\subset SO(4).$$ 
Using the definition (\ref{double cover}) of $\Phi,$ we easily get 
\begin{equation*}
\Phi^{-1}(SO(2)\times SO(2))=\{\cos \theta +\sin \theta\ I,\ \theta\in\mathcal{A}\}=:S^1_{\mathcal{A}}\hspace{.2cm}\subset Spin(4),
\end{equation*}
where the $\cos$ and $\sin$ functions are naturally defined on the Lorentz numbers by
$$\cos \theta\ =\ \sum_{n=0}^{+\infty}\frac{(-1)^n}{2n!}\theta^{2n}\hspace{.5cm}\mbox{and}\hspace{.5cm}\sin \theta\ =\ \sum_{n=0}^{+\infty}\frac{(-1)^n}{2n+1!}\theta^{2n+1}$$
for all $\theta\in\mathcal{A}.$  Indeed, setting $\theta=\frac{1+\sigma}{2}s+\frac{1-\sigma}{2}t,$ we have in fact
$$\cos \theta=\frac{1+\sigma}{2}\cos s+\frac{1-\sigma}{2}\cos t\hspace{.3cm}\mbox{and}\hspace{.3cm}\sin \theta=\frac{1+\sigma}{2}\sin s+\frac{1-\sigma}{2}\sin t,$$
and the usual trigonometric formulas hold for these trigonometric functions; then, setting $\theta=\theta_1+\sigma\theta_2$ with $\theta_1,\theta_2\in\R,$ we get, in $\HA,$
$$\cos \theta+\sin \theta\ I=(\cos\theta_2+\sigma\sin\theta_2\ I).(\cos\theta_1+\sin\theta_1\ I),$$
and $\Phi(\cos \theta +\sin \theta\ I)$ appears to be the rotation of $\R^{4}$ consisting in a rotation of angle $2\theta_1$ in $\{0\}\times\R^{2}$ and in a rotation of angle $2\theta_2$ in $\R^2\times\{0\}$. Summing up the preceding results, we define
\begin{equation}\label{def spin2}
Spin(2):=\{\cos\theta_1+\sin\theta_1\ I,\ \theta_1\in\R\}\subset Spin(4)
\end{equation}
and
\begin{equation}\label{def spin2'}
Spin'(2):=\{\cos\theta_2+\sigma\sin\theta_2\ I,\ \theta_2\in\R\}\subset Spin(4),
\end{equation}
and we have
$$S^1_{\mathcal{A}}=Spin'(2).Spin(2)\simeq Spin(2)\times Spin(2)/\Z_2,$$
and the double covering 
$$\Phi:\hspace{.5cm}S^1_{\mathcal{A}}=Spin'(2).Spin(2)\hspace{.5cm}\stackrel{2:1}{\longrightarrow} \hspace{.5cm}SO(2)\times SO(2).$$
If we now restrict the spin representation $\rho$ of $Spin(4)$ to $S^1_{\mathcal{A}}\subset Spin(4)$ the representation in $\HA=\HA_+\oplus\HA_-$ splits into four subspaces
\begin{equation}\label{splitting HA}
\HA_+=S^{++}\oplus S^{--}\hspace{1cm}\mbox{and}\hspace{1cm}\HA_-=S^{+-}\oplus S^{-+}
\end{equation}
where
$$S^{++}=\frac{1+\sigma}{2}\left(\R J\oplus\R K\right),\ S^{--}=\frac{1+\sigma}{2}\left(\R \unitquat\oplus\R I\right),\ S^{+-}=\frac{1-\sigma}{2}\left(\R \unitquat\oplus\R I\right)$$
and
$$S^{-+}=\frac{1-\sigma}{2}\left(\R J\oplus\R K\right).$$
\begin{rem}\label{rmk repr}
The representation
\begin{eqnarray}\label{def rho}
\rho:\hspace{.5cm}Spin'(2)\times Spin(2)&\rightarrow& End_{\C}(\HA)\label{rep spin2}\\
(g_1,g_2)&\mapsto&\rho(g):\xi\mapsto g\xi,\nonumber
\end{eqnarray}
where $g=g_1g_2\in S^1_{A}=Spin'(2).Spin(2)\ \subset\HA,$  is equivalent to the representation 
\begin{equation}\label{splitting rep spin}
\rho_1\otimes\rho_2\ =\ \rho_1^+\otimes\rho_2^+\ \ \oplus\ \ \rho_1^-\otimes\rho_2^-\ \ \oplus\ \ \rho_1^+\otimes\rho_2^-\ \ \oplus\ \ \rho_1^-\otimes\rho_2^+
\end{equation}
of $Spin(2)\times Spin(2),$ where $\rho_1=\rho_1^++\rho_1^-$ and $\rho_2=\rho_2^++\rho_2^-$ are two copies of the spinor representation of $Spin(2);$ moreover, the decomposition (\ref{splitting HA}) corresponds to the decomposition (\ref{splitting rep spin}): indeed, writing
$$\theta=\theta_1+\sigma\theta_2=\frac{1+\sigma}{2}\left(\theta_1+\theta_2\right)+\frac{1-\sigma}{2}\left(\theta_1-\theta_2\right),$$
it is not difficult to see that the restrictions of the representation (\ref{rep spin2}) to the subspaces $S^{++},$ $S^{--},$ $S^{+-}$ and $S^{-+}$ are respectively equivalent to the multiplications by $e^{-i(\theta_1+\theta_2)},$ $e^{i(\theta_1+\theta_2)},$ $e^{i(\theta_1-\theta_2)}$ and $e^{i(-\theta_1+\theta_2)}$ on $\C$ where $\theta_2\in\R/2\pi\Z$ describes the first factor and $\theta_1\in\R/2\pi\Z$ the second factor of $Spin'(2)\times Spin(2),$ as in (\ref{def spin2})-(\ref{def spin2'}). 
\end{rem}

\section{Twisted spinor bundle}\label{section twisted spinor bundle}

We assume that $M$ is an oriented surface, with a metric, and that $E\rightarrow M$ is a vector bundle of rank two, oriented, with a fibre metric and a connection compatible with the metric. We set $\Sigma:=\Sigma E\otimes\Sigma M,$ the tensor product of spinor bundles constructed from $E$ and $TM.$ We denote by $Q:=Q_E\times_M Q_M$ the $SO(2)\times SO(2)$ principal bundle on $M,$ product of the positive and orthonormal frame bundles on $E$ and $TM.$  If $\tilde{Q}:=\tilde{Q}_E\times_M\tilde{Q}_M$ is the $Spin(2)\times Spin(2)$ principal bundle product of the spin structures $\tilde{Q}_E\rightarrow Q_E$ and $\tilde{Q}_M\rightarrow Q_M,$ the bundle $\Sigma$ is associated to $\tilde{Q}$ and to the representation (\ref{def rho}),  that is
$$\Sigma\simeq\tilde{Q}\times\HA/\rho.$$ 
This is because $\rho$ is equivalent to the representation $\rho_1\otimes\rho_2,$ tensor product of two copies of the spin representation of $Spin(2)$ (see Remark \ref{rmk repr} above). Obviously, the maps  $\xi\mapsto g \xi$ belong in fact to $End_{\HA}(\HA),$ the space of endomorphisms of $\HA$ which are $\HA-$linear, where the linear structure on $\HA$ is given by the multiplication on the right: $\Sigma$ is thus naturally equipped with a linear right-action of $\HA.$ Since the products $\langle\langle.,.\rangle\rangle$ and $H(.,.)$ on $\HA$ are preserved by the multiplication of $Spin(4)$ on the left, $\Sigma$ is moreover naturally equipped with the products
$$\langle\langle.,.\rangle\rangle:\begin{array}[t]{rcl}\Sigma\times\Sigma&\rightarrow&\HA\\
(\varphi,\varphi')&\mapsto&\langle\langle\xi,\xi'\rangle\rangle\end{array}$$
and
$$H:\begin{array}[t]{rcl}\Sigma\times\Sigma&\rightarrow&\mathcal{A}\\
(\varphi,\varphi')&\mapsto&H(\xi,\xi'),\end{array}$$
where $\xi$ and $\xi'\ \in \HA$ are the coordinates of $\varphi$ and $\varphi'$ in some spinorial frame.
\\

We quote the following properties: for all $X\in E\oplus TM$ and $\varphi,\psi\in\Sigma,$ 
\begin{equation}\label{bracket property 1}
\langle\langle X\cdot\varphi,\psi\rangle\rangle=-\bigtch{\langle\langle \varphi,X\cdot\psi\rangle\rangle}
\end{equation}
and
\begin{equation}\label{bracket property 2}
\langle\langle\varphi,\psi\rangle\rangle=\overline{\langle\langle\psi,\varphi\rangle\rangle}.
\end{equation}
\textbf{Notation.}
Throughout the paper, we will use the following notation: if $\tilde{s}\in\tilde{Q}$ is a given frame, the brackets $[\cdot]$ will denote the coordinates $\in \HA$ of the spinor fields in $\tilde{s},$ that is, for $\varphi\in\Sigma,$
$$\varphi\simeq[\tilde{s},[\varphi]]\hspace{1cm}\in\hspace{.5cm} \Sigma\simeq\tilde{Q}\times\HA/\rho.$$
  We will also use the brackets to denote the coordinates in $\tilde{s}$ of the elements of the Clifford algebra $Cl(E\oplus TM)$: $X\in Cl_0(E\oplus TM)$ and  $Y\in Cl_1(E\oplus TM)$ will be respectively represented by $[X]$ and $[Y]\in\HA$ such that, in $\tilde{s},$ 
$$X\simeq \left(\begin{array}{cc}[X]&0\\0&\tch{[X]}\end{array}\right)\hspace{1cm}\mbox{and}\hspace{1cm}Y\simeq\left(\begin{array}{cc}0&[Y]\\\tch{[Y]}&0\end{array}\right).$$
We note that
$$[X\cdot\varphi]=[X][\varphi]\hspace{1cm}\mbox{and}\hspace{1cm}[Y\cdot\varphi]=[Y]\tch{[\varphi]}.$$
If $(e_0,e_1)$ and $(e_2,e_3)$ are positively oriented and orthonormal frames of $E$ and $TM,$ a frame $\tilde{s}\in\tilde{Q}$ such that $\pi(\tilde{s})=(e_0,e_1,e_2,e_3),$ where $\pi:\tilde{Q}\rightarrow Q_E\times_M Q_M$ is the natural projection onto the bundle of the orthonormal frames of $E\oplus TM$, will be called \textit{adapted to the frame} $(e_0,e_1,e_2,e_3);$ in such a frame, $e_0,$ $e_1,$ $e_2$ and $e_3\in Cl_1(E\oplus TM)$ are respectively represented by $\sigma\textit{1},$ $I,$ $J$ and $K\in\HA.$

\section{Spinor representation of surfaces in $\R^4$}\label{section spinor representation}

The aim of this section is to formulate the main results of \cite{BLR} using the formalism introduced in the previous sections. We keep the notation of the previous section, and recall that $\Sigma=\Sigma E\otimes\Sigma M$ is equipped with a natural connection
$$\nabla\ =\ \nabla^E\otimes id_{\Sigma M}\ +\ id_{\Sigma E}\otimes\nabla^M,$$
the tensor product of the spinor connections on $\Sigma E$ and on $\Sigma M,$ and also with a natural action of the Clifford bundle
$$Cl(E\oplus TM)\simeq Cl(E)\hat\otimes Cl(M);$$
see \cite{BLR}. This permits to define the Dirac operator $D$ acting on $\Gamma(\Sigma)$ by
$$D\varphi=e_2\cdot\nabla_{e_2}\varphi+e_3\cdot\nabla_{e_3}\varphi,$$
where $(e_2,e_3)$ is an orthonormal basis of $TM.$ We have the following:
\begin{prop}\label{prop fundamental xi}
Let $\vec{H}$ be a section of $E,$ and assume that $\varphi\in\Gamma(\Sigma)$ is such that
\begin{equation}\label{dirac equation}
D\varphi=\vec{H}\cdot\varphi\hspace{1cm}\mbox{ and }\hspace{1cm}H(\varphi,\varphi)=1.
\end{equation}
We define the $\HA$-valued $1$-form $\xi\in\Omega^1(M,\HA)$ by
\begin{equation}\label{def xi}
\xi(X):=\langle\langle X\cdot\varphi,\varphi\rangle\rangle\hspace{1cm}\in\hspace{.5cm}\HA
\end{equation}
for all $X\in TM,$ where the pairing $\langle\langle.,.\rangle\rangle:\Sigma\times\Sigma\rightarrow\HA$ is defined in the previous section. Then
\begin{enumerate}
\item the form $\xi$ satisfies
$$\xi=-\tch{\overline{\xi}},$$
and thus takes its values in $\R^{4}\subset\HA$ (see (\ref{R4intoHA}));
\item the form $\xi$ is closed:
$$d\xi=0.$$
\end{enumerate}
\end{prop}
\begin{proof}
The first part of the proposition is a consequence of (\ref{bracket property 1}) and (\ref{bracket property 2}), and the second part relies on the Dirac equation (\ref{dirac equation}); we refer to \cite{BLR} for the detailed proof of a very similar proposition. 
\end{proof}
We may rewrite Theorem 1 of \cite{BLR} for surfaces in $\R^4$ as follows:
\begin{thm}\label{th main result}
Suppose that $M$ is moreover simply connected. The following statements are equivalent.
\begin{enumerate}
\item There exists a spinor field $\varphi$ of $\Gamma(\Sigma)$ with $H(\varphi,\varphi)=1$ solution of the Dirac equation
$$D\varphi=\vec{H}\cdot\varphi.$$
\item There exists a spinor field $\varphi\in\Gamma(\Sigma)$ with $H(\varphi,\varphi)=1$ solution of 
$$\nabla_X\varphi=-\frac{1}{2}\sum_{j=2,3}e_j\cdot B(X,e_j) \cdot\varphi, $$
where $B:TM\times TM\lgra E$ is bilinear with $\frac{1}{2}\trace(B)=\vec{H},$ and where $(e_2,e_3)$ is an orthonormal basis of $TM$ at every point.
\item There exists an isometric immersion $F$ of $M$ into $\R^{4}$ with normal bundle $E$ and mean curvature vector $\vec{H}$.
\end{enumerate}
Moreover, $F=\int\xi,$ where $\xi$ is the closed 1-form on $M$ with values in $\R^{4}$ defined by
\begin{equation}\label{def xi th}
\xi(X):=\langle\langle X\cdot\varphi,\varphi\rangle\rangle\hspace{.5cm}\in\hspace{.5cm}\R^{4}\subset\HA
\end{equation}
for all $X\in TM.$ 
\end{thm}
The proof that (1) is equivalent to (3) is very simple: assuming first that $M$ is immersed in $\R^4,$ the spinor bundle of $\R^4$ restricted to $M$ identifies to $\Sigma=\Sigma E\otimes\Sigma M$ where $E$ is the normal bundle of the surface in $\R^4$, and the restriction to $M$ of the constant spinor field $\unitquat$ of $\R^4$ satisfies (1) (by the spinor Gauss formula, see \cite{BLR}); conversely, if $\varphi\in\Gamma(\Sigma)$ satisfies (1), it is easy to check that the formula $F=\int\xi,$ where $\xi$ is defined by (\ref{def xi th}), defines an isometric immersion with normal bundle $E$ and mean curvature vector $\vec{H};$ see \cite{Bay,BLR}.
\begin{rem}\label{rmk ident normal}
The map $X\in E\mapsto\langle\langle X\cdot\varphi,\varphi\rangle\rangle\in\R^4$ identifies $E$ with the normal bundle of the immersion; it preserves the metrics, the connections and the fundamental forms. See \cite{Bay,BLR}.
\end{rem}
\begin{rem}\label{rmk general idea}
Applications of the spinor representation formula in Sections \ref{section local description} and \ref{section flat tori S3} will rely on the following simple observation: assume that $F_o:M\hookrightarrow\R^4$ is an isometric immersion and consider $\varphi=\unitquat_{|M}$ the restriction to $M$ of the constant spinor field $\unitquat$ of $\R^4;$ if 
\begin{equation}\label{formule repr}
F=\int\xi,\hspace{1cm}\xi(X)=\langle\langle X\cdot\varphi,\varphi\rangle\rangle
\end{equation}
is the immersion given in the theorem, then $F\simeq F_o.$ This is in fact trivial since
\begin{equation}\label{ident trivial}
\xi(X)\ =\ \overline{[\varphi]}[X]\tch{[\varphi]}\ =\ [X]\ \simeq\ X
\end{equation}
in a spinorial frame $\tilde{s}$ of $\R^4$ which is above the canonical basis (in such a frame $[\varphi]=\pm\unitquat$). The representation formula (\ref{formule repr}), when written in moving frames adapted to the immersion, will give non trivial formulas.
\end{rem}
\section{The Gauss map of a surface in $\R^4$}\label{section gauss map}

We assume in this section that $M$ is an oriented surface immersed in $\R^4.$ We consider $\Lambda^2\R^{4},$ the vector space of bivectors of $\R^{4}$ endowed with its natural metric $\langle.,.\rangle.$ The Grassmannian of the oriented 2-planes in $\R^{4}$ identifies with the submanifold of unit and simple bivectors
$$\mathcal{Q}=\{\eta\in\Lambda^2\R^{4}:\ \langle \eta,\eta\rangle=1,\ \eta\wedge\eta=0\},$$
and the oriented Gauss map of $M$ with the map
$$G:\ M\rightarrow \mathcal{Q},\ p\mapsto G(p)=u_1\wedge u_2,$$
where $(u_1,u_2)$ is a positively oriented and orthonormal basis of $T_pM.$ The Hodge $*$ operator $\Lambda^2\R^{4}\rightarrow \Lambda^2\R^{4}$ is defined by the relation
\begin{equation}\label{i lambda2}
\langle *\eta,\eta'\rangle=\eta\wedge\eta'
\end{equation}
for all $\eta,\eta'\in \Lambda^2\R^{4},$ where we identify $\Lambda^4\R^{4}$ to $\R$ using the canonical volume element $e_0^o\wedge e_1^o\wedge e_2^o\wedge e_3^o$ on $\R^{4};$ here and below $(e_0^o,e_1^o,e_2^o,e_3^o)$ stands for the canonical basis of $\R^{4}$. It satisfies $*^2=id_{\Lambda^2\R^{4}}$ and thus $\sigma=*$ naturally defines a $\mathcal{A}$-module structure on $\Lambda^2\R^{4}:$ it is such that
$$a\eta=a_1\eta+a_2*\eta$$
for all $a=a_1+\sigma a_2\in\mathcal{A}$ and $\eta\in\Lambda^2\R^4.$ We also define
\begin{equation}\label{H lambda2}
H(\eta,\eta')=\langle \eta,\eta'\rangle+\sigma\ \eta\wedge\eta'\hspace{.5cm}\in\hspace{.5cm}\mathcal{A}
\end{equation}
for all $\eta,\eta'\in\ \Lambda^2\R^{4}.$ This is a $\mathcal{A}$-bilinear map on $\Lambda^2\R^{4}$ since, by (\ref{i lambda2}),
$$H(\sigma\ \eta,\eta')=H(\eta,\sigma\ \eta')=\sigma\ H(\eta,\eta')$$ 
for all $\eta,\eta'\in\Lambda^2\R^4,$ and we have
$$\mathcal{Q}=\{\eta\in\Lambda^2\R^{4}:\ H(\eta,\eta)=1\}.$$
The bivectors
$$E_1=e_2^o\wedge e_3^o,\ E_2=e_3^o\wedge e_1^o,\ E_3=e_1^o\wedge e_2^o$$
form a basis of $\Lambda^2\R^{4}$ as a module over $\mathcal{A};$ this basis is such that $H(E_i,E_j)=\delta_{ij}$ for all $i,j.$ Using the Clifford map defined Section \ref{section clifford map}, and identifying $\Lambda^2\R^{4}$ with the elements of order 2 of $Cl_0(4)\subset Cl(4)\subset\HA(2)$ (see (\ref{even elements})), we get
$$E_1=I,\ E_2=J,\ E_3=K$$
and
$$\Lambda^2\R^{4}=\{aI+bJ+cK\in\HA:(a,b,c)\in\mathcal{A}^3\};$$
moreover, using this identification, the Lorentz structure $\sigma$ and the quadratic map $H$ defined above on $\Lambda^2\R^{4}$ coincide with the natural Lorentz structure $\sigma$ and the quadratic map $H$ defined on $\HA$ (Section \ref{subsection quaternions}), and
\begin{equation}\label{Q S2A}
\mathcal{Q}=\{aI+bJ+cK:\ (a,b,c)\in\mathcal{A}^3,\ a^2+b^2+c^2=1\}\hspace{.5cm}=:\ S^2_{\mathcal{A}}.
\end{equation}
We note that, by the isomorphism (\ref{isom HA HH}),
\begin{equation}\label{S2A product}
S^2_{\mathcal{A}}\ \simeq\ S^2\times S^2.
\end{equation}
We define the \textit{cross product} of two vectors $\xi,\xi'\in\Im m\ \HA:=\mathcal{A}I\oplus\mathcal{A}J\oplus\mathcal{A}K$ by
$$\xi \times \xi':=\frac{1}{2}\left(\xi\xi'-\xi'\xi\right)\hspace{.5cm} \in\ \Im m\ \HA.$$ 
It is such that
$$\langle\langle \xi,\xi'\rangle\rangle=H(\xi,\xi')\unitquat+\xi \times \xi'.$$
\begin{lem}\label{lemma omega 0}
If $\xi\times\xi'=0$ where $\xi$ is invertible in $\HA,$ then
\begin{equation}\label{rel xi xip}
\xi'=\lambda\ \xi
\end{equation}
for some $\lambda\in\mathcal{A}.$ 
\end{lem}
\begin{proof}
Writing
$$\xi=\frac{1+\sigma}{2}\ \xi_++\frac{1-\sigma}{2}\ \xi_-\hspace{.5cm}\mbox{and}\hspace{.5cm}\xi'=\frac{1+\sigma}{2}\ \xi'_++\frac{1-\sigma}{2}\ \xi'_-$$
where $\xi_+,$ $\xi_-,$ $\xi'_+$ and $\xi'_-$ belong to $\Im m\ \HH,$ the condition $\xi\times\xi'=0$ is equivalent to the two conditions $\xi_+\times\xi'_+=0$ and $\xi_-\times\xi'_-=0,$ where the cross product is here the usual cross product in $\Im m\ \HH\simeq\R^3;$ this is thus also equivalent to  the fact that both $\xi_+,$ $\xi'_+$ and $\xi_-,$ $\xi'_-$ are linearly dependent in $\R^3.$ If $\xi$ is moreover invertible in $\HA,$ then $\xi_+$ and $\xi_-$ are not zero (see Lemma \ref{invertible HA} in the Appendix), and the result easily follows.
\end{proof}
We also define the \textit{mixed product} of three vectors $\xi,\xi',\xi''\in \Im m\ \HA$ by
$$[\xi,\xi',\xi'']:= H(\xi \times \xi',\xi'')\hspace{.5cm} \in\ \mathcal{A}.$$
The mixed product is a \textit{$\mathcal{A}$-valued volume form} on $\Im m\ \HA;$ it induces a natural \textit{$\mathcal{A}$-valued area form} $\omega_{\mathcal{Q}}$ on $\mathcal{Q}$ by
$${\omega_{\mathcal{Q}}}_p(\xi,\xi'):=[\xi,\xi',p]$$ 
for all $p\in\mathcal{Q}$ and all $\xi,\xi'\in T_p\mathcal{Q}.$ We now compute the pull-back by the Gauss map of the area form $\omega_Q:$
\begin{prop}\label{prop pull back}
We have
\begin{equation}\label{formula pull-back omega}
G^*\omega_{\mathcal{Q}}=(K+\sigma K_N)\ \omega_M,
\end{equation}
where $\omega_M$ is the area form, $K$ is the Gauss curvature and $K_N$ is the normal curvature of $M\hookrightarrow\R^4.$ Assuming moreover that
$$dG_{x_o}:\hspace{.3cm}T_{x_o}M\rightarrow  T_{G(x_0)}\mathcal{Q}$$
is one-to-one at some point $x_o\in M,$ then $K=K_N=0$ at $x_o$ if and only if the linear space $dG_{x_o}(T_{x_o}M)$ is some $\mathcal{A}$-line in $T_{G(x_o)}\mathcal{Q},$ i.e. 
$$dG_{x_o}(T_{x_o}M)\hspace{.2cm}=\hspace{.2cm}\{a\ U,\ a\in\mathcal{A}\},$$
where $U\in T_{G(x_0)}\mathcal{Q}\subset\HA$ is such that $H(U,U)=1.$
\end{prop}
\begin{proof}
Let $(e_2,e_3)$ be a positively oriented and orthonormal frame tangent to $M.$ By definition, we have
\begin{equation}\label{pull back first step}
G^*\omega_{\mathcal{Q}}(e_2,e_3)=H\left(dG(e_2)\times dG(e_3),G\right).
\end{equation}
Since $G=e_2\wedge e_3,$ we have
$$dG\ =\ B(e_2,.)\wedge e_3+e_2\wedge B(e_3,.)\ \simeq\ B(e_2,.)\cdot e_3+e_2\cdot B(e_3,.)$$
where $B:TM\times TM\rightarrow E$ is the second fundamental form of $M\hookrightarrow\R^4,$
and straightforward computations give
\begin{eqnarray*}
dG(e_2)\times dG(e_3)&=&\frac{1}{2}\left(dG(e_2)\cdot dG(e_3)-dG(e_3)\cdot dG(e_2)\right)\\
&=&\alpha\ e_2\cdot e_3  + \beta
\end{eqnarray*}
with
$$\alpha=B(e_2,e_3)\cdot B(e_2,e_3)-\frac{1}{2}(B(e_2,e_2)\cdot B(e_3,e_3)+B(e_3,e_3)\cdot B(e_2,e_2))$$
and
$$\beta=(B(e_2,e_2)-B(e_3,e_3))\cdot B(e_2,e_3)-B(e_2,e_3)\cdot(B(e_2,e_2)-B(e_3,e_3)).$$
Here the dot $\cdot$ stands for the Clifford product in $Cl(4),$ which, for elements in $Cl_0(4)\simeq\HA,$ coincides with the product in $\HA.$ Writing
$$B\ =\ \left(\begin{array}{cc}
a&c\\
c&b\end{array}\right)e_0
+
\left(\begin{array}{cc}
e&g\\
g&f\end{array}\right)e_1$$
in the basis $(e_2,e_3),$ we then easily obtain
$$\alpha=ab+ef-c^2-g^2\hspace{.5cm}\mbox{and} \hspace{.5cm}\beta=((a-b)g-(e-f)c)\ e_0\cdot e_1.$$
The first term is $\alpha=K$ and the second term is
$$\beta=-\langle \left(S_{e_0}\circ S_{e_1}-S_{e_1}\circ S_{e_0}\right)(e_2),e_3\rangle\ e_0\cdot e_1=K_N\ e_0\cdot e_1,$$
where, for $\nu\in E,$ $S_{\nu}:TM\rightarrow TM$ is such that
$$\langle S_{\nu}(X),Y\rangle\ =\ \langle B(X,Y),\nu\rangle$$
for all $X,Y\in TM.$ This finally gives
\begin{equation}\label{wedge dG K KN}
dG(e_2)\times dG(e_3)=K\ e_2\cdot e_3+K_N\ e_0\cdot e_1.
\end{equation}
Since $e_2\cdot e_3=G$ and $e_0\cdot e_1=\sigma\ e_2\cdot e_3=\sigma G,$ (\ref{pull back first step}) and $H(G,G)=1$ imply (\ref{formula pull-back omega}).

We finally prove the last claim of the proposition: if $K=K_N=0,$ then formula (\ref{wedge dG K KN}) implies that $dG(e_2)\times dG(e_3)=0;$ writing
$$dG(e_2)=\frac{1+\sigma}{2}\ \xi_++\frac{1-\sigma}{2}\ \xi_-\hspace{.5cm}\mbox{and}\hspace{.5cm}dG(e_3)=\frac{1+\sigma}{2}\ \xi'_++\frac{1-\sigma}{2}\ \xi'_-$$
with $\xi_+,\xi'_+,\xi_-,\xi'_-\in\Im m\ \HH,$ this implies that $\xi_+,\xi'_+$ and $\xi_-,\xi'_-$ are linearly dependent (Lemma \ref{lemma omega 0} and its proof). We note that $\xi_+=\xi'_+=0$ or $\xi_-=\xi'_-=0$ are not possible since it would contradict that $dG$ is one-to-one. We deduce that $dG(T_{x_0}M)$ contains an element invertible in $\HA:$ assuming that $dG(e_2)$ and $dG(e_3)$ are not invertible, if $\xi_+\neq 0$ we get $\xi_-=0$ and thus $\xi'_-\neq 0,$ which in turn implies $\xi'_+=0;$ thus
$$\lambda\ dG(e_2)+(1-\lambda)\ dG(e_3)= \frac{1+\sigma}{2}\ \lambda\ \xi_++\frac{1-\sigma}{2}\ (1-\lambda)\ \xi'_-$$
yields an invertible element in $\HA$ for $\lambda\neq 0,1$ (see Lemma \ref{invertible HA} in the appendix). We denote by $u\in T_{x_0}M$ a vector such that $dG(u)$ is invertible in $\HA.$ If $v$ is another vector belonging to $T_{x_0}M,$ (\ref{wedge dG K KN}) implies that $dG(u)\times dG(v)=0$ and thus that
\begin{equation}\label{rel lin}
dG(v)=\lambda\ dG(u)
\end{equation}
for some $\lambda\in\mathcal{A}$ (Lemma \ref{lemma omega 0}). Thus $dG(T_{x_0}M)$ belongs to $\{a\ U',\ a\in\mathcal{A}\},$ where $U':=dG(u)\in\HA$ is invertible. Finally, considering $\mu\in\mathcal{A}$ such that $\mu^2=H(U',U')$ (Lemma \ref{lem square}) and setting $U:=\mu^{-1}U',$ we get $H(U,U)=\unitquat$ and $dG(T_{x_0}M)\subset \{a\ U,\ a\in\mathcal{A}\}.$ This is an equality since $dG$ is one-to-one.

Conversely, if $dG(T_{x_0}M)$ is some $\mathcal{A}$-line in $T_{G(x_0)}\mathcal{Q},$ we obviously get $G^*\omega_{\mathcal{Q}}=0$ at $x_0,$ and (\ref{formula pull-back omega}) implies that $K=K_N=0.$
\end{proof}
\begin{rem}
As a corollary of the proposition, we easily obtain the Gauss-Bonnet and the Whitney formulas:  integrating (\ref{formula pull-back omega}), we get 
\begin{equation}\label{integral formula}
\int_MG^*\omega_{\mathcal{Q}}=\int_M(K+\sigma K_N)\ \omega_M.
\end{equation}
Writing $\omega_Q=\frac{1+\sigma}{2}\ \omega_1+\frac{1-\sigma}{2}\ \omega_2$ and $G=\frac{1+\sigma}{2}\ G_1+\frac{1-\sigma}{2}\ G_2,$ we easily get
$$G^*\omega_Q=\frac{1+\sigma}{2}\ G_1^*\omega_1+\frac{1-\sigma}{2}\ G_2^*\omega_2.$$
Since the real forms $\omega_1,$ $\omega_2$ are in fact the usual area forms on each one of the two factors $S^2$ of the splitting (\ref{S2A product}), we get $\int_MG_1^*\omega_1=4\pi d_1$ and $\int_MG_2^*\omega_2=4\pi d_2$ where $d_1$ and $d_2$ are the degrees of $G_1:M\rightarrow S^2$ and $G_2:M\rightarrow S^2$ respectively, and formula (\ref{integral formula}) thus yields 
$$\int_M(K+\sigma K_N)\ \omega_M=4\pi\left(\frac{1+\sigma}{2}d_1+\frac{1-\sigma}{2}d_2\right),$$
which gives
$$\int_MK\ \omega_M=2\pi(d_1+d_2)\hspace{.5cm}\mbox{and}\hspace{.5cm}\int_MK_N\ \omega_M=2\pi(d_1-d_2).$$
See \cite{HO} and \cite{W2} for other proofs and further consequences of these formulas.
\end{rem}
We finally give the expression of the Gauss map when the immersion is given by a spinor field $\varphi\in\Gamma(\Sigma),$ as in Theorem \ref{th main result}:
\begin{lem}\label{etap2}The Gauss map of the immersion defined by $\varphi$ is given by 
\begin{equation}\label{G function g}
G=g^{-1}Ig
\end{equation}
where $g=[\varphi]$ in some local section of $\tilde{Q.}$ In this formula
$$G:M\rightarrow S^2_{\mathcal{A}}\hspace{.2cm}\subset \Im m\ \HA\hspace{.5cm}\mbox{and}\hspace{.5cm}g:M\rightarrow S^3_{\mathcal{A}}\hspace{.2cm}\subset\HA$$
are viewed as maps with values in the quaternions $\HA;$ see (\ref{Q S2A}) and (\ref{spin4 S3A}).
\end{lem}
\begin{proof}
We assume that $\varphi=[\tilde{s},g]\in\Sigma\simeq\tilde{Q}\times\HA/\rho$ (see Section \ref{section twisted spinor bundle}), and we denote by $(e_2,e_3)$ the positively oriented and orthonormal basis tangent to the immersion which is associated to $\tilde{s}.$ We first note that $G=\langle\langle e_2\cdot e_3\cdot\varphi,\varphi\rangle\rangle.$ Indeed, 
$$v_2:=\langle\langle e_2\cdot\varphi,\varphi\rangle\rangle\hspace{.5cm}\mbox{and}\hspace{.5cm}v_3:=\langle\langle e_3\cdot\varphi,\varphi\rangle\rangle$$
form a positive and orthonormal basis tangent to the immersion (see Theorem \ref{th main result}), and the Gauss map $G=v_2\wedge v_3$ identifies to 
\begin{equation}\label{ident clifford}
v_2\cdot v_3=\left(\begin{array}{cc}0&v_2\\\tch{v_2}&0\end{array}\right)\left(\begin{array}{cc}0&v_3\\\tch{v_3}&0\end{array}\right)=\left(\begin{array}{cc}v_2\tch{v_3}&0\\0&\tch{v_2}v_3\end{array}\right)\simeq v_2\tch{v_3}
\end{equation}
(identification of $\Lambda^2\R^4$ with the elements of order 2 of $Cl_o(\R^4)\subset \HA(2)$, as in the beginning of the section). Since $v_2=\overline{[\varphi]}[e_2]\tch{[\varphi]}$ and $v_3=\overline{[\varphi]}[e_3]\tch{[\varphi]}$ we get
$$G=v_2\tch{v_3}=\overline{[\varphi]}[e_2]\tch{[e_3]}[\varphi]=\langle\langle e_2\cdot e_3\cdot\varphi,\varphi\rangle\rangle.$$
Thus
\begin{equation}\label{G gIg}
G=\overline{g}Ig
\end{equation}
since $[\varphi]=g$ and $[e_2\cdot e_3]=I$ in $\tilde{s};$ this gives (\ref{G function g}) since $\overline{g}=g^{-1}$ ($g\in Spin(4)$).
\end{proof}
\begin{rem}
The lemma implies that the diagram 
$$\xymatrix{
  &\Sigma E\otimes\Sigma M\ar[dl]^\pi \ar[d]_p\\
   M\ar@/^/[ur]^\varphi\ar[r]_{G} & M\times\mathcal{Q}
  }$$
is commutative, where the projection $p$ is defined by
\begin{eqnarray*}
 p:\hspace{.5cm}\Sigma E\otimes\Sigma M&\rightarrow& M\times\mathcal{Q}\\
 \varphi_x&\mapsto& (x,\ [\varphi_x]^{-1}\ I\ [\varphi_x])
 \end{eqnarray*}
(here  $\varphi_x$ belongs to the fibre above the point $x\in M,$ and $[\varphi_x]$ denotes its component in some frame belonging to $\tilde{Q}_x$), and where the horizontal map is $x\mapsto (x,G(x)).$ Thus $\varphi$ is a lift of the Gauss map (in fact a horizontal lift, if we consider the connection $\nabla\varphi-\eta\cdot\varphi$ with $\eta:=-1/2\sum_{j=2,3}e_j\cdot B(.,e_j)$). Theorem \ref{th main result} thus roughly says that some lift of the Gauss map permits to write the immersion intrinsically. We will get a nicer picture in the case of a flat immersion with flat normal bundle; see Lemma \ref{rmk lift} and Proposition \ref{prop diag} below. 
\end{rem}
\section{Lorentz surfaces and Lorentz numbers}\label{section lorentz surfaces}

In this section we present elementary results concerning Lorentz surfaces and Lorentz numbers; we refer to \cite{CMMS,L} for an exposition of the theory in the framework of para-complex geometry.
\\

We will say that a surface $M$ is a Lorentz surface if there is a covering by open subsets $M=\cup_{\alpha\in S}U_{\alpha}$ and charts
$$\varphi_{\alpha}:\hspace{.3cm}U_{\alpha}\ \rightarrow\ \mathcal{A},\hspace{.5cm} \alpha\in S$$ 
such that the transition functions 
$$\varphi_{\beta}\circ\varphi_{\alpha}^{-1}:\hspace{.3cm}\varphi_{\alpha}(U_\alpha\cap U_\beta)\subset\mathcal{A}\ \rightarrow\ \varphi_{\beta}(U_\alpha\cap U_\beta)\subset\mathcal{A},\hspace{.5cm}\alpha,\ \beta\in S$$
are conformal maps in the following sense: for all $a\in\varphi_{\alpha}(U_\alpha\cap U_\beta)$ and $h\in\mathcal{A},$
$$d\ (\varphi_{\beta}\circ\varphi_{\alpha}^{-1})_a\ (\sigma\ h)\hspace{.3cm}=\hspace{.3cm}\sigma\ d\ (\varphi_{\beta}\circ\varphi_{\alpha}^{-1})_a\ (h).$$
A Lorentz structure is also equivalent to a smooth family of maps 
$$\sigma_x:\hspace{.3cm}T_xM\ \rightarrow\ T_xM,\hspace{.5cm} \mbox{with}\hspace{.5cm} \sigma_x^2=id_{T_xM},\ \sigma_x\neq\pm id_{T_xM}.$$
This definition coincide with the definition of a Lorentz surface given in \cite{Weinstein}: a Lorentz structure is equivalent to a conformal class of Lorentz metrics on the surface, that is to a smooth family of cones in every tangent space of the surface, with distinguished lines. Indeed, the cone at $x\in M$ is
$$Ker(\sigma_x-id_{T_xM})\ \cup\ Ker(\sigma_x+id_{T_xM})$$
where the sign of the eigenvalues $\pm 1$ permits to distinguish one of the lines from the other. 

If $M$ is moreover oriented, we will say that the Lorentz structure is compatible with the orientation of $M$ if the charts $\varphi_\alpha: U_{\alpha}\rightarrow\mathcal{A},\ \alpha\in S$ preserve the orientations (the positive orientation in $\mathcal{A}=\{x+\sigma y,\ x,y\in\R\}$ is naturally given by $(\partial_x,\partial_y$)). In that case, the transition functions are conformal maps $\mathcal{A}\rightarrow\mathcal{A}$ preserving orientation.

If $M$ is a Lorentz surface, a smooth map $\psi:M\rightarrow \mathcal{A}$ (or $\mathcal{A}^n,$ or a Lorentz surface) will be said to be a conformal map if $d\psi$ preserves the Lorentz structures, that is if
$$d\psi_x(\sigma_xu)\ =\ \sigma_{\psi(x)}(d\psi_x(u))$$
for all $x\in M$ and $u\in T_xM.$ In a chart $\mathcal{A}=\{x+\sigma y,\ x,y\in\R\},$ a conformal map satisfies 
\begin{equation}\label{eqn crl}
\frac{\partial \psi}{\partial y}\ =\ \sigma\ \frac{\partial \psi}{\partial x}.
\end{equation}
Defining the coordinates $(s,t)$ such that
\begin{equation}\label{def s t}
x+\sigma\ y\ =\ \frac{1+\sigma}{2}\ s+\frac{1-\sigma}{2}\ t
\end{equation}
($s$ and $t$ are parameters along the distinguished lines) and writing
$$\psi\ =\ \frac{1+\sigma}{2}\ \psi_1+\frac{1-\sigma}{2}\ \psi_2$$
with $\psi_1,\psi_2\in\R,$ (\ref{eqn crl}) reads
$$\partial_t\psi_1=\partial_s\psi_2=0,$$
and we get 
$$\psi_1=\psi_1(s)\hspace{1cm}\mbox{and}\hspace{1cm} \psi_2=\psi_2(t).$$
A conformal map is thus equivalent to two functions of one variable. We finally note that if $\psi:M\rightarrow\mathcal{A}^n$ is a conformal map, we have, in a chart $a:\mathcal{U}\subset\mathcal{A}\rightarrow M,$
$$d\psi\ =\ \psi' da,$$
where $da=dx+\sigma dy$ and $\psi'$ belongs to $\mathcal{A}^n;$ this is a direct consequence of (\ref{eqn crl}).

As a consequence of Proposition \ref{prop pull back}, if $K=K_N=0$ and if $G:M\rightarrow\mathcal{Q}$ is a regular map (i.e. if $dG_x$ is injective at every point $x$ of $M$), there is a unique Lorentz structure $\sigma$ on $M$ such that
\begin{equation}\label{def lorentz structure}
dG_x(\sigma\ u)\ =\ \sigma\ dG_x(u)
\end{equation}
for all $x\in M$ and all $u\in T_xM.$ This is because 
$$dG_x(T_xM)\ =\ \{a\ U,\ a\in\mathcal{A}\}$$
for some $U\in T_{G(x)}\mathcal{Q}\subset \HA$ in that case, which implies that $dG_x(T_xM)$ is stable by multiplication by $\sigma.$ This Lorentz structure is the Lorentz structure introduced in \cite{GM}.

\section{Local description of flat surfaces with flat normal bundle in $\R^4$}\label{section local description}

In this section we suppose that $M$ is simply connected and that the bundles $TM$ and $E$ are flat ($K=K_N=0).$ We recall that the bundle $\Sigma:=\Sigma E\otimes \Sigma M$ is associated to the principal bundle $\tilde{Q}$ and to the representation $\rho$ of the structure group $Spin(2)\times Spin(2)$ in $\HA$ given by (\ref{def rho}). Since the curvatures $K$ and $K_N$ are zero, the spinorial connection on the bundle $\tilde{Q}$ is flat, and $\tilde{Q}$ admits a parallel local section $\tilde{s};$ since $M$ is simply connected, the section $\tilde{s}$ is in fact globally defined.
\\

We consider $\varphi\in\Gamma(\Sigma)$ a solution of
\begin{equation}\label{eqn dirac}
D\varphi=\vec H\cdot\varphi
\end{equation}
such that $H(\varphi,\varphi)=1,$ and $g=[\varphi]:M\rightarrow Spin(4)$ the coordinates of $\varphi$  in $\tilde{s}:$ 
$$\varphi=\left[\tilde{s},g\right]\hspace{.5cm}\in\hspace{.2cm}\Sigma=\tilde{Q}\times\HA/\rho.$$
Note that, by Theorem \ref{th main result}, $\varphi$ also satisfies
$$\nabla_X\varphi=\eta(X)\cdot\varphi$$
for all $X\in TM,$ where $\eta(X)=-1/2\sum_{j=2,3}e_j\cdot B(X,e_j)$ for some bilinear map $B:TM\times TM\rightarrow E.$ 
\\

In the following, we will denote by $(e_0,e_1)$ and $(e_2,e_3)$ the parallel, orthonormal and positively oriented frames, respectively normal, and tangent to $M,$  corresponding to $\tilde{s}$ (i.e. such that $\pi(\tilde{s})=(e_0,e_1,e_2,e_3)$ where $\pi:\tilde{Q}\rightarrow Q_E\times Q_M$ is the natural projection).
\\

We moreover assume that the Gauss map $G$ of the immersion defined by $\varphi$ is regular, and consider the Lorentz structure $\sigma$ induced on $M$ by $G,$ defined by (\ref{def lorentz structure}). 
\\

We now show that $g$ is in fact a conformal map admitting a special parametri\-zation, its \emph{arc length}, and that, in such a special parametrization, $g$ depends on a single conformal map $\psi:\mathcal{U}\subset\mathcal{A}\rightarrow\mathcal{A}.$ To state the theorem, we define
$$\mathbb{G}:=\{a\mapsto \pm a+b,\ b\in\mathcal{A}\},$$
subgroup of transformations of $\mathcal{A}.$
\begin{thm}\label{thm chart A}
Under the hypotheses above, we have the following:
\\
\\1- The map $g:M\rightarrow S^3_{\mathcal{A}}\subset\HA$ is a conformal map, and, at each point of $M,$ there is a local chart $a:\mathcal{U}\subset\mathcal{A}\rightarrow M,$ unique up to the action of $\mathbb{G},$ which is compatible with the orientation of $M$ and such that $g:\mathcal{U}\subset\mathcal{A}\rightarrow S^3_{\mathcal{A}}$ satisfies
$$H(g',g')\equiv \unitquat,$$
where $g':\mathcal{U}\subset\mathcal{A}\rightarrow \HA$ is such that $dg=g'da.$
\\
\\2- There exists a conformal map $\psi:\mathcal{U}\subset\mathcal{A}\rightarrow\mathcal{A}$ such that
\begin{equation}\label{etap theta}
g'g^{-1}=\cos\psi\ J+\sin\psi\ K,
\end{equation}
where $a:\mathcal{U}\subset\mathcal{A}\rightarrow M$ is a chart defined in 1-. 
\end{thm}
\begin{rem}\label{rem interpretation a psi}
The local chart $a$ may be interpreted as the arc length of $g,$ and the function $\psi'$ as the geodesic curvature of $g:\mathcal{U}\subset\mathcal{A}\rightarrow S^3_{\mathcal{A}}$: indeed  (\ref{etap theta}) expresses that $\psi$ is the angle of the derivative of the "curve" $g$ in the trivialization $TS^3_{\mathcal{A}}\simeq S^3_{\mathcal{A}}\times\mathcal{A}^3$ with respect to the fixed basis $J,K.$
\end{rem}
For the proof of the theorem, we will need the following lemmas:
\begin{lem}\label{etap1} Denoting by $[\eta]\in\Omega^1(M,\HA)$ the 1-form which represents $\eta$ in $\tilde{s},$ we have
\begin{equation}\label{AJ+AK}
[\eta]=dg\ g^{-1}=\eta_1J+\eta_2 K,
\end{equation}
where $\eta_1$ and $\eta_2$ are 1-forms on $M$ with values in $\mathcal{A}.$
\end{lem}
\begin{proof}
Since $\varphi=[\tilde{s},g]$ and $\nabla\varphi=[\tilde{s},dg]$ ($\tilde{s}$ is a parallel section of $\tilde{Q}$), the equation $\nabla\varphi=\eta\cdot\varphi$ reads $[\eta]=dg\ g^{-1},$ where $[\eta]\in\Omega^1(M,\HA)$ represents $\eta$ in $\tilde{s}.$ Since $\eta(X)=-\frac{1}{2}\sum_{j=2,3}e_j\cdot B(X,e_j),$ we have
$$[\eta(X)]=-\frac{1}{2}\sum_{j=2,3}[e_j]\tch{[B(X,e_j)]}$$
with $[e_2]=J,$ $[e_3]=K$ and $[B(X,e_j)]\in\ \R\sigma\unitquat\oplus\R I,$ $j=2,3$ (since $B(X,e_j)$ is normal to the surface (see the end of Section \ref{section twisted spinor bundle})), and (\ref{AJ+AK}) follows.
\end{proof}
\begin{lem}\label{etap2p}
The form $\tilde{\eta}=\langle\langle\eta\cdot\varphi,\varphi\rangle\rangle\in\Omega^1(M,\Im m\ \HA)$ is given by
\begin{equation}\label{expr eta g}
\tilde{\eta}=\frac{1}{2}G^{-1}dG=g^{-1}dg.
\end{equation}
\end{lem}
\noindent\textit{Proof of Lemma \ref{etap2p}:} Recalling Lemma \ref{etap2}, we have
\begin{equation}\label{G function g2}
G=g^{-1}Ig=\overline{g}Ig.
\end{equation}
Differentiating this identity, we get 
$$dG=d\overline{g}Ig+\overline{g}Idg.$$
Since $G^{-1}=\overline{G}=-G$ ($G$ belongs to $S^2_{\mathcal{A}}=S^3_{\mathcal{A}}\cap\Im m\ \HA$), we get
\begin{equation}\label{G gIg interm}
G^{-1}dG=-\overline{g}Igd\overline{g}Ig-\overline{g}Ig\overline{g}Idg.
\end{equation}
The second term is $\overline{g}dg.$ To analyze the first term, we note that $gd\overline{g}=\overline{dg\overline{g}}$ is a linear combination of $J$ and $K$ (Lemma \ref{etap1}) and thus anticommutes with $I;$ the first term in (\ref{G gIg interm}) is thus 
$$\overline{g}gd\overline{g}IIg=-d\overline{g}g=\overline{g}dg$$ 
(using that $\overline{g}g=\unitquat$ for the last equality), and equation (\ref{G gIg interm}) thus yields
$$\frac{1}{2}G^{-1}dG=\overline{g}dg=g^{-1}dg,$$
which gives the second identity in (\ref{expr eta g}). Finally, this is the form $\tilde{\eta}$ since
\begin{equation}\label{etatilde etap}
\tilde{\eta}=\langle\langle\eta\cdot\varphi,\varphi\rangle\rangle=g^{-1}[\eta]g
\end{equation}
with $[\eta]=dg\ g^{-1}$ (Lemma \ref{etap1}). 
$\finpreuve$

Formula (\ref{G function g}) in Lemma \ref{etap2} together with the special form of (\ref{AJ+AK}) may be rewritten as follows:
\begin{lem}\label{rmk lift}
Consider the projection
\begin{eqnarray*}
p:\hspace{.5cm}Spin(4)\hspace{.2cm}\subset\ \HA&\rightarrow &\mathcal{Q}\hspace{.2cm}\subset\ \Im m\ \HA\\
g&\mapsto& g^{-1}Ig
\end{eqnarray*}
as a $S^1_{\mathcal{A}}$ principal bundle, where the action of $S^1_{\mathcal{A}}$ on $Spin(4)$ is given by the multiplication on the left. This fibration is formally analogous to the classical Hopf fibration $S^3\subset\HH\rightarrow S^2\subset \Im m\ \HH,\ g\mapsto g^{-1}Ig$. It is equipped with the horizontal distribution given at every $g\in Spin(4)$ by
$$\mathcal{H}_g:={d(R_{g^-1})_g}^{-1}\left(\mathcal{A} J\oplus\mathcal{A} K\right)\hspace{1cm}\subset\hspace{.5cm}T_gSpin(4)$$
where $R_{g^-1}$ stands for the right-multiplication by $g^{-1}$ on $Spin(4).$ The distribution $(\mathcal{H}_g)_{g\in Spin(4)}$ is $H$-orthogonal to the fibers of $p,$ and, for all $g\in Spin(4),$ $dp_g:\mathcal{H}_g\rightarrow T_{p(g)}Q$ is an isomorphism which preserves $\sigma$ and such that 
$$H(dp_g(u),dp_g(u))=4H(u,u)$$
for all $u\in\mathcal{H}_g.$ With these notations, we have
\begin{equation}\label{relation G g}
G=p\circ g,
\end{equation}
and the map $g:M\rightarrow Spin(4)$ appears to be a horizontal lift to $Spin(4)$ of the Gauss map $G:M\rightarrow\mathcal{Q}$ (formulas (\ref{G function g}) and (\ref{etap theta})).
\end{lem}

\noindent\textit{Proof of Theorem \ref{thm chart A}:} 
Let $a:\mathcal{U}\subset\mathcal{A}\rightarrow M$ be a chart given by the Lorentz structure induced by $G$ and compatible with the orientation of $M.$ By Lemma \ref{rmk lift}, $g:\mathcal{U}\subset\mathcal{A}\rightarrow S^3_{\mathcal{A}}$ is a conformal map (since so are $G$ and $p$ in (\ref{relation G g})). We consider $g'$ such that $dg=g'da.$ If $\mu:\mathcal{A}\rightarrow\mathcal{A}$ is a conformal map, we have $H((g\circ\mu)',(g\circ\mu)')=\mu'^2H(g',g').$ We observe that we may find $\mu,$ unique up to the action of $\mathbb{G},$  such that $\mu'^2H(g',g')\equiv \unitquat.$ To this end, we first note the following
\begin{lem}\label{gp inversible}
The map $g:M\rightarrow S^3_{\mathcal{A}}\ \subset\HA$ is an immersion and $g'$ is thus invertible in $\HA.$
\end{lem}
\begin{proof}
By (\ref{expr eta g}), $g$ is an immersion since so is $G.$ Assume by contradiction that $g'$ is not invertible in $\HA;$ $g'$ would belong to $\HA_+\cup\HA_-$ (Lemma \ref{invertible HA}), and we would have (since $\partial_y=\sigma\partial_x$ and $g$ is conformal)
$$dg(\partial_y)=dg(\sigma\partial_x)=\sigma dg(\partial_x)=\sigma g'=\pm g'=\pm dg(\partial_x),$$
which contradicts that $dg$ is injective. 
\end{proof}
Thus $H(g',g')$ is invertible in $\mathcal{A},$ and its inverse is of the form $H(\xi,\xi)$ for some $\xi$ invertible in $\HA.$ Lemma \ref{lem square} thus gives $\mu',$ invertible in $\mathcal{A},$ such that $\mu'^2=H(g',g')^{-1}.$ There are four solutions, $\pm\mu',\pm\sigma\mu',$ but only two of them define after integration a conformal map $\mu:\mathcal{A}\rightarrow\mathcal{A}$ preserving orientation. We then obtain $\mu$ by integration, unique up to the action of the group $\mathbb{G}.$ We finally note that $\mu$ is a diffeomorphism ($\mu'$ is invertible in $\mathcal{A}$, which implies that $d\mu$ is an isomorphism), and, considering $g\circ\mu$ instead of $g,$ we may thus assume that $H(g',g')=1,$ as claimed in the theorem. Writing 
$$g=\ \frac{1+\sigma}{2}\ g_1\ +\ \frac{1-\sigma}{2}\ g_2$$
with $g_1=g_1(s)\in\HH$ and $g_2=g_2(t)\in\HH$ ($g$ is a conformal map), we get
$$g'g^{-1}\ =\ \frac{1+\sigma}{2}\ {g_1}'g_1^{-1}\ +\ \frac{1-\sigma}{2}\ {g_2}'g_2^{-1},$$
with $H({g_1}'g_1^{-1},{g_1}'g_1^{-1})=H({g_2}'g_2^{-1},{g_2}'g_2^{-1})=1.$ Since ${g_1}'g_1^{-1}$ and ${g_2}'g_2^{-1}$ belong to $\R J\oplus\R K,$ we deduce that
\begin{equation}\label{g1 g2 psi1 psi2}
{g_1}'g_1^{-1}=\cos\psi_1\ J+\sin\psi_1\ K\hspace{.5cm}\mbox{and}\hspace{.5cm}{g_2}'g_2^{-1}=\cos\psi_2\ J+\sin\psi_2\ K
\end{equation}
for $\psi_1=\psi_1(s)$ and $\psi_2=\psi_2(t)\in\R.$ The function
\begin{equation}\label{def psi1 psi2}
\psi\ =\ \frac{1+\sigma}{2}\ \psi_1(s)\ +\ \frac{1-\sigma}{2}\ \psi_2(t)
\end{equation}
satisfies (\ref{etap theta}). $\finpreuve$
\begin{rem}\label{rmk interpretation psi 1 2}
Similarly to Remark \ref{rem interpretation a psi}, the functions $\psi_1'$ and $\psi_2',$ where $\psi_1$ and $\psi_2$ are defined by (\ref{def psi1 psi2}), may be interpreted as the geodesic curvatures of $g_1:(s_1,s_2)\rightarrow S^3$ and $g_2:(t_1,t_2)\rightarrow S^3$.
\end{rem}
The aim is now to study the metric of the surface in the special chart $a=x+\sigma y$ adapted to $g,$ given by Theorem \ref{thm chart A}. We recall that $(e_0,e_1)$ and $(e_2,e_3)$ are the parallel, orthonormal and positively oriented frames, respectively normal, and tangent to $M,$  corresponding to $\tilde{s}.$ Let us write
$$\vec H\ =\ h_0e_0+h_1e_1.$$
We suppose that $\psi:\mathcal{U}\subset\mathcal{A}\rightarrow\mathcal{A}$ is the conformal map defined above, and we write
$$\psi\ =\ \theta_1+\sigma\theta_2$$
with $\theta_1$ and $\theta_2\in\R.$
\begin{lem}\label{lemma ei partiali}
We have
\begin{equation}\label{e_2 e_3 function lambda mu theta_1}
\left\{\begin{array}{lcr}
e_2&=&\sin\theta_1\ \frac{1}{\lambda}\ \partial_x+\cos\theta_1\ \frac{1}{\mu}\ \partial_y\\
e_3&=&-\cos\theta_1\ \frac{1}{\lambda}\ \partial_x+\sin\theta_1\ \frac{1}{\mu}\ \partial_y
\end{array}\right.
\end{equation}
where $\lambda,\mu\in\R^*$ satisfy 
\begin{equation}\label{def mu lambda}
\left(\begin{array}{r}\frac{1}{\mu}\\\frac{1}{\lambda}\end{array}\right)= \left(\begin{array}{rr}\cos\theta_2&\sin\theta_2\\-\sin\theta_2&\cos\theta_2\end{array}\right) \left(\begin{array}{c}h_0\\h_1\end{array}\right).
\end{equation}
Moreover, we have $\lambda\mu>0.$
\end{lem}
\begin{proof}
In the chart $a:\mathcal{U}\subset\mathcal{A}\rightarrow M$ introduced above, $e_2,e_3$ are represented by two functions $\underline{e_2},\underline{e_3}:\mathcal{U}\subset\mathcal{A}\rightarrow\mathcal{A}.$ In $\tilde{s},$ the Dirac equation (\ref{eqn dirac}) reads
$$[e_2]\ \tch{[\nabla_{e_2}\varphi]}+[e_3]\ \tch{[\nabla_{e_3}\varphi]}\ =\ [\vec H]\ \tch{[\varphi]},$$
that is
$$J\ dg(e_2)+K\ dg(e_3)=(-\sigma h_0\unitquat+h_1I)\ g;$$
since $dg(e_2)g^{-1}=g'g^{-1}\underline{e_2}$ and $dg(e_3)g^{-1}=g'g^{-1}\underline{e_3},$ using (\ref{etap theta}) this may be written
\begin{equation}\label{eqn dirac local}
\left(\begin{array}{rr}\cos\psi & \sin\psi\\ \sin\psi & -\cos\psi\end{array}\right)\left(\begin{array}{c}\sigma h_0\\h_1\end{array}\right)=\left(\begin{array}{c}\underline{e_2}\\\underline{e_3}\end{array}\right).
\end{equation}
Setting $c:=h_0\cos\theta_2+h_1\sin\theta_2$ and $d:=-h_0\sin\theta_2+h_1\cos\theta_2,$
(\ref{eqn dirac local}) reads
$$\underline{e_2}=d\sin\theta_1+\sigma c\cos\theta_1\hspace{.5cm}\mbox{and}\hspace{.5cm}\underline{e_3}=-d\cos\theta_1+\sigma c\sin\theta_1.$$
Since $\underline{e_2}$ and $\underline{e_3}$ represent the independent vectors $e_2,e_3,$ we have $cd\neq 0;$ setting $\mu=1/c$ and $\lambda=1/d,$ we get (\ref{e_2 e_3 function lambda mu theta_1}) and (\ref{def mu lambda}). Moreover, $\lambda\mu>0$ since the bases $(e_2,e_3)$ and $(\partial_x,\partial_y)$ are both positively oriented.
\end{proof}
\begin{prop}\label{prop metric system} In the chart $a=x+\sigma y$ of Theorem \ref{thm chart A}, the metric reads
\begin{equation}\label{metric lambda mu}
\lambda^2\ dx^2\ +\ \mu^2\ dy^2;
\end{equation}
moreover, $\lambda$ and $\mu$ are solutions of the hyperbolic system
\begin{equation}\label{systeme hyperbolique}
\left\{\begin{array}{lcr}\partial_x\mu&=&\lambda\ \partial_x\theta_2\\
\partial_y\lambda&=&-\mu\ \partial_y\theta_2.
\end{array}\right.
\end{equation}
\end{prop}
\begin{rem}
Note that
$$\partial_x\theta_2=\partial_y\theta_1\hspace{.5cm}\mbox{ and }\hspace{.5cm}\partial_y\theta_2=\partial_x\theta_1$$
(since $\psi=\theta_1+\sigma\theta_2$ is a conformal map), and thus that (\ref{systeme hyperbolique}) is equivalent to
\begin{equation}\label{systeme hyperbolique2}
\left\{\begin{array}{lcr}\partial_x\mu&=&\lambda\ \partial_y\theta_1\\
\partial_y\lambda&=&-\mu\ \partial_x\theta_1.
\end{array}\right.
\end{equation}
Hyperbolic systems similar to (\ref{systeme hyperbolique}) and  (\ref{systeme hyperbolique2}) appear in \cite{GM}.
\end{rem}
\begin{proof}
We just write that the basis $(e_2,e_3),$ given in $(\partial_x,\partial_y)$ by (\ref{e_2 e_3 function lambda mu theta_1}), is orthonormal and parallel: since the basis $(e_2,e_3)$ is orthonormal, we get $I=P^t\ A P,$ where $A$ is the matrix of the metric in $(\partial_x,\partial_y),$ and where 
$$P=\left(\begin{array}{cc}\frac{1}{\lambda}\sin\theta_1&-\frac{1}{\lambda}\cos\theta_1\\\frac{1}{\mu}\cos\theta_1&\frac{1}{\mu}\sin\theta_1\end{array}\right)$$
is the matrix of change of bases, and thus
$$A=\left(\begin{array}{cc}\lambda^2&0\\0&\mu^2\end{array}\right),$$
which is (\ref{metric lambda mu}). We then compute the Christoffel symbols using the Christoffel formulas, and easily get
$$\Gamma_{xx}^x=\frac{1}{\lambda}\partial_x\lambda,\hspace{.5cm} \Gamma_{yx}^x=\frac{1}{\lambda}\partial_y\lambda,\hspace{.5cm} \Gamma_{xy}^y=\frac{1}{\mu}\partial_x\mu,\hspace{.5cm} \Gamma_{yy}^y=\frac{1}{\mu}\partial_y\mu$$
and
$$ \Gamma_{xx}^y=-\frac{\lambda}{\mu^2}\partial_y\lambda,\hspace{.5cm} \Gamma_{yy}^x=-\frac{\mu}{\lambda^2}\partial_x\mu.$$
Writing that $(e_2,e_3),$ given by (\ref{e_2 e_3 function lambda mu theta_1}), is parallel with respect to the metric (\ref{metric lambda mu}), we get (\ref{systeme hyperbolique2}), and thus also (\ref{systeme hyperbolique}).
\end{proof}
\begin{rem}\label{rmk ff} (Relation to the second fundamental form). If we consider the orthonormal basis normal to the surface
\begin{equation}\label{u_0 u_1 function e_0 e_1}
u_0=\cos\theta_2\ e_0+\sin\theta_2\ e_1\hspace{.5cm}\mbox{ and }\hspace{.5cm}u_1=-\sin\theta_2\ e_0+\cos\theta_2\ e_1,
\end{equation}
and the orthonormal basis tangent to the surface
$$u_2=\frac{1}{\lambda}\partial_x\hspace{.5cm} \mbox{ and }\hspace{.5cm}u_3= \frac{1}{\mu}\partial_y,$$  
the second fundamental form $B$ is given by
\begin{equation}\label{scd ff}
B(u_2,u_2)=\frac{2}{\lambda}u_1,\hspace{.5cm} B(u_3,u_3)=\frac{2}{\mu}u_0,\hspace{.5cm} B(u_2,u_3)=0.
\end{equation}
This may be proved by direct computations using the relation 
$$B(X,Y)=X\cdot \eta(Y)-\eta(Y)\cdot X$$ 
(see \cite{BLR}) together with (\ref{etap theta}), (\ref{AJ+AK}) and (\ref{e_2 e_3 function lambda mu theta_1}). 
\end{rem}
\begin{rem}\label{asymptotic net}
The special chart in Theorem \ref{thm chart A} generalizes the asymptotic Tchebychef net of a flat surface in $S^3.$ Indeed, we may easily obtain that the metric of the surface is given by
$$\frac{\lambda^2+\mu^2}{4}\ \left(ds^2+dt^2\right)\ +\ \frac{\lambda^2-\mu^2}{2}\ dsdt$$ 
and the second fundamental form by
$$B=\frac{1}{2}\left(\lambda u_1+\mu u_0\right)\left(ds^2+dt^2\right)+\left(\lambda u_1-\mu u_0\right)dsdt$$ 
in the coordinates $(s,t)$ defined in (\ref{def s t}). For a flat surface in $S^3$ these formulas reduce in fact to
$$ds^2+dt^2-2\cos 2\theta_2\ dsdt\hspace{.5cm}\mbox{and}\hspace{.5cm}<B,e_1>=-2\sin 2\theta_2\ ds dt$$
respectively, where, if $e_0$ stands for the outer unit normal vector of $S^3,$ the vector $e_1$ is such that $(e_0,e_1)$ is a positively oriented and orthonormal basis normal to $M;$ this is because $\lambda=-2\sin\theta_2$ and $\mu=2\cos\theta_2$ in that case (see Remark \ref{metric torus S3} below). These last formulas are characteristic of an asymptotic Tchebychef net \cite{Mo,K}.
\end{rem}
We finally obtain the local structure of surfaces with $K=K_N=0$ and regular Gauss map (see \cite{CD,DT} for the first description, and \cite{GM} for another representation):
\begin{thm}\label{thm local description}  Let $\psi:\mathcal{U}\subset\mathcal{A}\rightarrow\mathcal{A}$ be a conformal map, and $\theta_1,\theta_2:\mathcal{U}\rightarrow\R$ be such that 
$$\psi=\theta_1+\sigma\theta_2;$$ 
suppose that $\lambda,\mu$ are solutions of (\ref{systeme hyperbolique}) such that $\lambda\mu>0$, and define 
\begin{equation}\label{def alpha 2 3}
\underline{e_2}=\sin\theta_1\frac{1}{\lambda}+\sigma\cos\theta_1\frac{1}{\mu}\hspace{.5cm}\mbox{and}\hspace{.5cm}\underline{e_3}=-\cos\theta_1\frac{1}{\lambda}+\sigma\sin\theta_1\frac{1}{\mu}. 
\end{equation}
Then, if $g:\mathcal{U}\rightarrow Spin(4)\subset\HA$ is a conformal map solving 
$$g'g^{-1}=\cos\psi\ J+\sin\psi\ K,$$ 
and if we set
\begin{equation}\label{formula g omegai}
\xi:=g^{-1}(\omega_2J+\omega_3K)\hat{g}
\end{equation}
where $\omega_2,\omega_3:T\mathcal{U}\rightarrow\R$ are the dual forms of $\underline{e_2},\underline{e_3}\in\Gamma(T\mathcal{U}),$ the function $F=\int\xi$ defines an immersion $\mathcal{U}\rightarrow\R^{4}$ with $K=K_N=0.$ Reciprocally, the immersions of $M$ such that $K=K_N=0$ and with regular Gauss map are locally of this form.
\end{thm}
\begin{proof}
We first prove the direct statement. We consider the metric $\lambda^2dx^2+\mu^2dy^2$ on $\mathcal{U};$ it is such that $e_2\simeq \underline{e_2},$ $e_3\simeq \underline{e_3}$ $\in\Gamma(T\mathcal{U})$ defined by (\ref{def alpha 2 3}) form an orthonormal frame of $T\mathcal{U}.$ Since $(\lambda,\mu)$ is a solution of (\ref{systeme hyperbolique}), the frame $(e_2,e_3)$ is parallel (see Proposition \ref{prop metric system} and its proof), and the metric is flat. We also consider the trivial bundle $E:=\R^2\times \mathcal{U},$ with its trivial metric and its trivial connection; the canonical basis $(e_0,e_1)$ of $\R^2$ thus defines orthonormal and parallel sections $\in\Gamma(E)$. We set $s:=(e_0,e_1,e_2,e_3)\in Q=(SO(2)\times SO(2))\times \mathcal{U},$ and $\tilde{s}\in\tilde{Q}=(Spin(2)\times Spin(2))\times \mathcal{U}$ such that $\pi(\tilde{s})=s,$ where $\pi:\tilde{Q}\rightarrow Q$ is the natural covering (Section \ref{section splitting}). We then consider $\varphi\in\Sigma=\tilde{Q}\times\HA/\rho$ such that $[\varphi]=g$ in $\tilde{s}.$ The form $\xi=\langle\langle X\cdot\varphi,\varphi\rangle\rangle$ is a closed 1-form (since $\varphi$ is a solution of the Dirac equation $D\varphi=\vec{H}\cdot\varphi,$ where $\vec{H}=h_0e_0+h_1e_1$ is defined by (\ref{def mu lambda}); see the proof of Lemma \ref{lemma ei partiali} and Proposition \ref{prop fundamental xi}), and $F=\int\xi$ is an isometric immersion of $M$ in $\R^4$ whose normal bundle identifies to $E.$ Thus it is a flat immersion in $\R^4,$ with flat normal bundle. 

Reciprocally, if $F:M\rightarrow\R^4$ is the immersion of a flat surface with flat normal bundle and regular Gauss map, we have
$$F=\int\xi,\hspace{1cm}\mbox{with}\hspace{.5cm}\xi(X)=\langle\langle X\cdot\varphi,\varphi\rangle\rangle,$$
where $\varphi$ is the restriction to $M$ of the constant spinor field $\unitquat$ of $\R^4$ (Remark \ref{rmk general idea}). In a parallel frame $\tilde{s},$ we have $\varphi=[\tilde{s},g],$ where $g:M\rightarrow Spin(4)\subset\HA$ is an horizontal and conformal map (Lemma \ref{rmk lift}). In a chart compatible with the Lorentz structure induced by the Gauss map and adapted to $g$ (Theorem \ref{thm chart A}), $\xi$ is of the form (\ref{formula g omegai}) where $(\omega_1,\omega_2)$ is the dual basis of the basis defined by (\ref{def alpha 2 3}) and where in this last expression $\lambda,\mu$ are solutions of (\ref{systeme hyperbolique}) (Proposition \ref{prop metric system}).
\end{proof}
As a corollary, we obtain a new proof of the following result \cite{DT}:
\begin{cor}\label{cor local description}
Locally, a flat surface with flat normal bundle and regular Gauss map depends on 4 functions of one variable.
\end{cor}
\begin{proof}
We first observe that the function $\psi$ depends on two functions of one variable: since $\psi:\mathcal{A}\rightarrow \mathcal{A}$ is a conformal map, writing
$$\psi=\frac{1+\sigma}{2}\ \psi_1+\frac{1-\sigma}{2}\ \psi_2,$$
we have
$$\psi_1=\psi_1(s)\hspace{1cm}\mbox{ and }\hspace{1cm}\psi_2=\psi_2(t)$$
where the coordinates $(s,t)$ are defined by (\ref{def s t}). Now, writing the system (\ref{systeme hyperbolique}) in the coordinates $(s,t)$ we get
\begin{equation}\label{systeme hyperbolique s t}
\partial_s\left(\begin{array}{c}\lambda\\\mu\end{array}\right)=\left(\begin{array}{cc}+1&0\\0&-1\end{array}\right)\partial_t\left(\begin{array}{cc}\lambda\\\mu\end{array}\right)-\frac{1}{2}\left(\begin{array}{cc}0&\psi'_1+\psi'_2\\\psi'_2-\psi'_1&0\end{array}\right)\left(\begin{array}{c}\lambda\\\mu\end{array}\right);
\end{equation}
we may solve a Cauchy problem for this system: once $\psi_1$ and $\psi_2$ are given, a solution of (\ref{systeme hyperbolique s t}) depends on two functions $\mu(0,t),$ $\lambda(0,t)$ of the variable $t.$ By Theorem \ref{thm local description}, the surface depends on $\psi_1(s),\psi_2(t),\mu(0,t)$ and $\lambda(0,t).$
\end{proof}
\begin{rem} We may try to solve the hyperbolic system (\ref{systeme hyperbolique}), prescribing additionally the conformal class of the metric: setting
$$\lambda^2dx^2+\mu^2dy^2=e^{2f}\left(\lambda_o^2dx^2+\mu_o^2dy^2\right)$$
where $\lambda_o$ and $\mu_o$ are given functions, (\ref{systeme hyperbolique}), in the form (\ref{systeme hyperbolique s t}), easily gives 
\begin{equation}\label{systeme hyperbolique f s t}
\partial_s f=A\hspace{.5cm}\mbox{and}\hspace{.5cm}\partial_t f=B
\end{equation}
where $A$ and $B$ depend on the functions $\lambda_o,\mu_o,\psi_1,\psi_2$ and their first derivatives. The existence of a solution thus relies on the compatibility condition $\partial_tA=\partial_sB,$ which depends on the Gauss map $G$ up to its third derivatives (since the second derivatives of $\psi$ depend on $g$ up to its third derivatives, and $g$ is a horizontal lift of $G$). If this condition holds, the solution $f$ is unique, up to a constant: in the given conformal class, the metric is thus unique up to homothety. This is a special case of the general problem concerning the prescription of the Gauss map, whose solvability relies on a condition of order 3 on the prescribed Gauss map; see \cite{HO,W2}. 
\end{rem}

\section{The structure of the flat tori in $S^3$} \label{section flat tori S3}
In this section we obtain a spinorial proof of the description of the flat tori in $S^3$ \cite{B,Sa,S,K}: a flat torus in $S^3$ is necessarily a product of two horizontal curves in $S^3\subset\HH.$ We determine in the first subsection the global Lorentz structure induced on a flat torus in $S^3,$ we then study the global structure of the map which represents a constant spinor field in a parallel frame adapted to the torus, we then give a representation formula for the flat tori in $S^3$ and finally prove that they are always products of two horizontal closed curves in $S^3;$ in the last two subsections, we link this description to the classical Kitagawa representation of flat tori in $S^3$, and also deduce the structure of the Gauss map image. 

\subsection{The Lorentz structure induced by the Gauss map}\label{section flat tori R4}
We suppose here that $M$ is an oriented flat torus immersed in $S^3;$ its normal bundle $E\rightarrow M$ is trivial, and has a natural orientation. We note that the bundles of frames $Q_E$ and $Q_M$ are trivial, since there exist globally defined positively oriented and orthonormal frames $(e_0,e_1)$ and $(e_2,e_3),$ which are respectively normal and tangent to $M;$ in the following, we moreover assume that these frames are parallel ($K=K_N=0$). We then consider two spin structures
$$\tilde{Q}_E\rightarrow Q_E\hspace{1cm}\mbox{and}\hspace{1cm}\tilde{Q}_M\rightarrow Q_M$$ 
such that, setting $\tilde{Q}=\tilde{Q}_E\times_M\tilde{Q}_M,$ the spinor bundle $\Sigma\R^4_{| M}$ identifies to the bundle $\Sigma=\Sigma E\otimes\Sigma M =\tilde{Q}\times\HA/\rho$ where $\rho$ is defined by (\ref{def rho}). We note that the Gauss map of $M$ is regular (since $M$ is in $S^3$), and we suppose that $M$ is endowed with the Lorentz structure induced by its Gauss map (see Section \ref{section lorentz surfaces}) and compatible with its orientation. 
\begin{prop}\label{prop unif} 
Let us denote by $\pi:\tilde{M}\rightarrow M$ the universal covering of $M,$ and consider the Lorentz structure on $\tilde{M}$ such that $\pi$ is a conformal map. Then $\tilde{M}$ is conformal to $\mathcal{A},$ and $M$ is conformal to $\mathcal{A}/\Gamma,$ where  $\Gamma$ is a subgroup of translations of $\mathcal{A}\simeq \R^{2}$ generated by two elements 
\begin{equation}\label{structure gamma 1}
\Gamma=<\tau_1,\tau_2>,
\end{equation}
where the translations $\tau_1=a\mapsto a+(s_1,t_1)$ and $\tau_2=a\mapsto a+(s_2,t_2)$ are such that \begin{equation}\label{structure gamma 2}
s_1\Z\oplus s_2\Z=S\Z \hspace{1cm}\mbox{and}\hspace{1cm} t_1\Z\oplus t_2\Z=T\Z
\end{equation}
for some positive numbers $S,T$ (the coordinates $(s,t)$ are defined in (\ref{def s t})). Moreover, the lines $s$ and $t$ are closed in the quotient $\mathcal{A}/\Gamma.$
\end{prop}

\begin{proof}
We first consider a local section $\tilde{s}\in\Gamma(\tilde{Q})$ which is a local lift of the parallel frame $s=(e_0,e_1,e_2,e_3)\in\Gamma(Q),$ and the locally defined map $g:M\rightarrow S^3_{\mathcal{A}}$ such that
$$\varphi=[\tilde{s},g]\hspace{.5cm}\in\hspace{.3cm} \Sigma=\tilde{Q}\times\HA/\rho$$
where $\varphi=\unitquat_{|M}\in\Gamma(\Sigma).$ By Theorem \ref{thm chart A}, there exists a chart $a:\mathcal{U}\subset\mathcal{A}\rightarrow M$ (the arc length of $g$) compatible with the orientation and such that
$$g'g^{-1}=\cos\psi\ J+\sin\psi\ K$$
for some conformal map $\psi:\mathcal{U}\subset\mathcal{A}\rightarrow\mathcal{A};$ this chart is moreover unique up to the action of $\mathbb{G}=\{a\mapsto \pm a+b,\ b\in\mathcal{A}\}.$ We note that in fact this chart does not depend on the choice of the local section $\tilde{s},$ since, if we choose the other local lift of $s,$  the spinor field $\varphi$ is represented by $-g$ instead of $g.$ Using these very special charts, we may consider the Lorentz structure on $M$ as a $(\mathbb{G},X)$-structure, with $X=\mathcal{A}$ and $\mathbb{G}=\{a\mapsto \pm a+b,\ b\in\mathcal{A}\}.$ We consider its developping map
$$D:\tilde{M}\rightarrow X=\mathcal{A},$$
its holonomy representation $h:\gamma\in\pi_1(M)\rightarrow g_{\gamma}\in \mathbb{G}$ and its holonomy group $\Gamma=h(\pi_1(M));$ from the general theory of the $(\mathbb{G},X)$-structures (see \cite{T}), there is a map $\overline{D}:M\rightarrow X/\Gamma=\mathcal{A}/\Gamma
$ such that the diagram
$$\xymatrix{
  \tilde{M}\ar[r]^D\ar[d]_\pi  & \mathcal{A} \ar[d]\\
   M\ar[r]^{\overline{D}} & \mathcal{A}/\Gamma
  }$$
commutes. Moreover, if $\tilde{M}$ is endowed with the Lorentz structure induced by $\pi,$ the developing map $D$ is by construction a conformal map. We note that the $(\mathbb{G},X)$-structure is complete since $M$ is compact (Prop. 3.4.10 in \cite{T}). This implies that $D$ identifies $\tilde{M}$ to $\mathcal{A}$ and that $\overline{D}$ identifies $M$ to $\mathcal{A}/\Gamma$ in the diagram above (p. 142 and Prop. 3.4.5 in \cite{T}). We will implicitly do these identifications below.

The group $\Gamma$ is commutative and generated by two elements, since so is $\pi_1(M);$ moreover, rank $\Gamma=2$ since the quotient $\mathcal{A}/\Gamma$ is compact. We then easily deduce that $\Gamma$ is generated by $\tau_1=a\mapsto a+b_1,$ $\tau_2=a\mapsto a+b_2,$ where $b_1=(s_1,t_1)$ and $b_2=(s_2,t_2)$ (in the coordinates $(s,t)$) are independent vectors of $\mathcal{A}\simeq \R^{2}.$ 

We now consider the Gauss map $G:\mathcal{A}/\Gamma\rightarrow S^2_{\mathcal{A}}$ and its lift to the universal covering $\tilde{G}:\mathcal{A}\rightarrow S^2_{\mathcal{A}}.$ Since $\tilde{G}$ is a conformal map, it is of the form
\begin{equation}\label{dec g tilde}
\tilde{G}=\frac{1+\sigma}{2}\ \tilde{G}_1+\frac{1-\sigma}{2}\ \tilde{G}_2
\end{equation}
with $\tilde{G}_1=\tilde{G}_1(s)$ and $\tilde{G}_2=\tilde{G}_2(t)\in S^2.$ Moreover, $\tilde{G}$ is $\Gamma$-invariant and thus satisfies: $\forall\ a\in\mathcal{A},$ $\forall\ p,q\in\Z,$
$$\tilde{G}(a+p(s_1,t_1)+q(s_2,t_2))=\tilde{G}(a),$$
that is
\begin{equation}\label{periodicity G1 G2}
\tilde{G}_1(s+ps_1+qs_2)=\tilde{G}_1(s)\hspace{.5cm}\mbox{and}\hspace{.5cm}\tilde{G}_2(t+pt_1+qt_2)=\tilde{G}_2(t);
\end{equation}
thus the subgroups $s_1\Z\oplus s_2\Z$ and $t_1\Z\oplus t_2\Z$ are not dense in $\R$ (since the maps $\tilde{G}_1$ and  $\tilde{G}_2$ are not constant ($G$ is an immersion, and so is $\tilde{G}$)) and thus are of the form (\ref{structure gamma 2}) for some positive numbers $S$ and $T.$ We now show that the line $s$ is closed in the quotient $\mathcal{A}/\Gamma.$ We consider $(p,q)\in\Z^2\backslash(0,0)$ such that $0=pt_1+qt_2$ (if $t_2\neq0,$ (\ref{structure gamma 2}) implies that $t_1/t_2$ belongs to $\Q$). Since $ps_1+qs_2$ belongs to $S\Z,$ there exists $m\in\Z\backslash\{0\}$ such that  $ps_1+qs_2=mS$  (note that $m\neq 0$ since $(s_1,t_1)$ and $(s_2,t_2)$ are linearly independent). Finally,
$$m(S,0)=p(s_1,t_1)+q(s_2,t_2)$$ 
belongs to $\Gamma,$ and the result follows. Similarly, the line $t$ is closed in $\mathcal{A}/\Gamma.$
\end{proof}
\begin{rem}\label{rmk periodic}
The curves $\tilde{G}_1$ and $\tilde{G}_2$ in (\ref{dec g tilde}) are periodic, with period $S$ and $T$ respectively (see (\ref{periodicity G1 G2}), together with the definition (\ref{structure gamma 2})). 
\end{rem}
\begin{rem}
The form of condition (\ref{structure gamma 2}) exactly means that the lines $s$ and $t$ are closed in the quotient $\mathcal{A}/\Gamma.$ Indeed, let us suppose that the line $s$ is closed in $\mathcal{A}/\Gamma,$ and consider $(p,q)\in\Z^2\backslash(0,0)$ such that 
\begin{equation}\label{cond s1 s2 t1 t2}
ps_1+qs_2=S'\hspace{1cm}\mbox{and}\hspace{1cm}pt_1+qt_2=0,
\end{equation}
for some $S'\in\R\backslash\{0\}.$ We will prove the second condition in (\ref{structure gamma 2}). We may assume without loss of generality that $p$ and $q$ are relatively prime numbers, or equivalently that 
$$\alpha p+\beta q=1$$
for some $\alpha,\beta\in\Z.$ Setting $T=\frac{t_1}{q}=-\frac{t_2}{p}$ (if $p$ or $q$ is 0, then $t_1$ or $t_2$ is 0 (by (\ref{cond s1 s2 t1 t2})) and the second condition in (\ref{structure gamma 2}) is trivial), we have $t_1=qT$ and $t_2=-pT,$ and thus $t_1\Z\oplus t_2\Z\subset T\Z;$ moreover,   
$$T=(\alpha p+\beta q)T=-\alpha t_2+\beta t_1\in t_1\Z\oplus t_2\Z,$$
and thus $t_1\Z\oplus t_2\Z=T\Z,$ which is the second condition in (\ref{structure gamma 2}).
\end{rem}
\subsection{The structure of the map $g$}
Since the bundle $\tilde{Q}\rightarrow M$ introduced at the beginning of the previous section is maybe not trivial, we consider the universal covering $\pi:\tilde{M}\simeq \mathcal{A}\rightarrow M\simeq\mathcal{A}/\Gamma,$ and the pull-backs 
$$\pi^*\tilde{Q}\ \simeq\ \pi^*\tilde{Q}_E\times_{\tilde{M}}\pi^*\tilde{Q}_M\hspace{1cm}\mbox{and}\hspace{1cm}\pi^*\Sigma\ \simeq\ \pi^*\tilde{Q}\times\HA/\rho.$$
The bundle $\pi^*\tilde{Q}\rightarrow\tilde{M}$ is trivial and admits a global section $\tilde{s}$ which is a lift of the parallel frame $(e_0,e_1,e_2,e_3)\in\pi^*Q_E\times_{\tilde{M}}\pi^*Q_M$ (since $\tilde{M}$ is simply connected).  We then define the map $g:\mathcal{A}\rightarrow S^3_{\mathcal{A}}$ such that
$$\varphi=[\tilde{s},g]\hspace{.5cm}\in\hspace{.3cm} \pi^*\Sigma\ \simeq\ \pi^*\tilde{Q}\times\HA/\rho$$
where $\varphi=\pi^*\unitquat_{|M}\in\Gamma(\pi^*\Sigma)$ is the constant spinor field $\unitquat$ of $\R^4,$ restricted to $M$ and pulled-back to the universal covering of $M.$ It appears that the map $g$ is globally an horizontal lift of the Gauss map, with a very simple structure. We consider
$$G:\mathcal{A}/\Gamma\rightarrow S^2_{\mathcal{A}}\hspace{1cm}\mbox{and}\hspace{1cm}\tilde{G}:\mathcal{A}\rightarrow S^2_{\mathcal{A}},$$
the Gauss map of $M\simeq\mathcal{A}/\Gamma,$ and its lift to the universal covering; they are conformal maps by construction. We also consider the Hopf fibration 
\begin{eqnarray*}
p:\hspace{.5cm} S^3_{\mathcal{A}}\hspace{.3cm}\subset\ \HA&\rightarrow& S^2_{\mathcal{A}}\hspace{.3cm}\subset\ \Im m\ \HA\\ 
g&\mapsto& g^{-1}Ig,
\end{eqnarray*} 
with its natural horizontal distribution (see Lemma \ref{rmk lift}). 
\begin{prop}\label{prop diag}
The function $g:\mathcal{A}\rightarrow S^3_{\mathcal{A}}$ is a horizontal lift of $G$ and $\tilde{G}:$ the diagram
$$\xymatrix{
  \mathcal{A}\ar[d]_{\pi}\ar[r]^{g}\ar[rd]^{\tilde{G}}   & S^3_{\mathcal{A}} \ar[d]^p \\
    \mathcal{A}/\Gamma \ar[r]_G & S^2_{\mathcal{A}}
  }$$
commutes. Moreover, there exists a conformal map $\psi:\mathcal{A}\rightarrow\mathcal{A}$ such that
$$g'{g}^{-1}=\cos\psi\ J+\sin\psi\ K.$$
\end{prop}
\begin{proof}
We have $\tilde{G}=g^{-1}Ig$ where $g'g^{-1}$ belongs to $\mathcal{A}J\oplus\mathcal{A}K$ (see Lemmas \ref{etap2} and \ref{etap1}). This is the first part of the proposition.  Moreover, the existence of $\psi:\mathcal{A}\rightarrow\mathcal{A}$ may be proved exactly as in the proof of (\ref{etap theta}) in Theorem \ref{thm chart A}. 
\end{proof}
We finally give the structure of $g$ (recall Proposition \ref{prop unif} and the definition (\ref{structure gamma 2}) of $S$ and $T$):
\begin{prop}\label{prop structure g}
We have
\begin{equation}\label{dec g}
g=\frac{1+\sigma}{2}\ g_1+\frac{1-\sigma}{2}\ g_2
\end{equation}
with $g_1=g_1(s)$ and $g_2=g_2(t)\in S^3.$ Moreover, one of the following two situations occurs:
\begin{enumerate}
\item $g_1$ and $g_2$ are periodic curves with period $S$ and $T$ respectively;
\item $g_1$ and $g_2$ satisfy
\begin{equation}\label{anti periodicity}
g_1(s+kS)=(-1)^kg_1(s)\hspace{1cm}\mbox{and}\hspace{1cm}g_2(t+kT)=(-1)^kg_2(t)
\end{equation}
for all $s,t\in\R$ and $k\in\Z;$ in that case the subgroup $\Gamma$ is such that
\begin{equation}\label{condition Gamma}
\Gamma\subset\{(mS,nT),\ m,n\in\Z,\ m\equiv n\ [2]\}.
\end{equation}
\end{enumerate}
\end{prop}
\begin{proof}
Identity (\ref{dec g}) is a consequence of the fact that $g$ is a conformal map. Since $\varphi=[\tilde{s},g]$ is $\Gamma$-invariant ($\varphi$ is the pull-back of a section of a bundle on $\mathcal{A}/\Gamma$) and since $\tilde{s}$ is a lift of the $\Gamma$-invariant frame $(e_0,e_1,e_2,e_3),$  we get
$$g(a+\gamma)=\varepsilon_{\gamma}\ g(a)$$
for all $a\in\mathcal{A}$ and all $\gamma\in\Gamma,$ where $\varepsilon:\Gamma\rightarrow\{\pm 1\}.$ In view of the definition (\ref{structure gamma 2}) of $S$ and $T,$ this easily gives
\begin{equation}\label{g1 g2 pm periodic}
g_1(s+S)=\pm g_1(s)\hspace{1cm}\mbox{and}\hspace{1cm}g_2(t+T)=\pm g_2(t)
\end{equation}
for all $s,t\in\R.$ The map $\xi(e_0):=\langle\langle e_0\cdot\varphi,\varphi\rangle\rangle:\mathcal{A}\rightarrow\HA$ is $\Gamma$-invariant (since $e_0$ and $\varphi$ are pull-backs of sections of bundles on $\mathcal{A}/\Gamma$), and reads, in $\tilde{s},$
\begin{equation}\label{e_0 tilde s}
\xi(e_0)=\langle\langle e_0\cdot\varphi,\varphi\rangle\rangle=\overline{[\varphi]}[e_0]\tch{[\varphi]}= \sigma g^{-1}\tch{g},
\end{equation}
since $e_0$ is represented by $[e_0]=\sigma\unitquat$ in $\tilde{s}.$ Using (\ref{dec g}), (\ref{e_0 tilde s}) reads
$$\xi(e_0)=\frac{\sigma}{2}\left(\overline{g_1}g_2+\overline{g_2}g_1\right)+\frac{1}{2}\left(\overline{g_1}g_2-\overline{g_2}g_1\right)\simeq \overline{g_1}g_2,$$
where, for the last identification, we use
\begin{eqnarray}
\R^4\subset\HA &\simeq&\HH\label{ident quat}\\
\sigma q_0\unitquat+q_1I+q_2 J+q_3 K&\simeq &q_0\unitquat+q_1I+q_2 J+q_3 K\nonumber
\end{eqnarray}
which identifies $\frac{\sigma}{2}\left(q+\overline{q}\right)+\frac{1}{2}\left(q-\overline{q}\right)\in\R^4\subset\HA$ to $q\in\HH.$ Writing that $\xi(e_0)$ is $\Gamma$-invariant, we then deduce
\begin{equation*}
\overline{g_1}(s+ps_1+q s_2)\ g_2(t+pt_1+qt_2)\ =\ \overline{g_1}(s)g_2(t)
\end{equation*}
for all $s,t\in\R$ and all $p,q\in\Z,$ and thus
\begin{equation}\label{relation periodicity g1g2}
\overline{g_1}(s+kS)g_2(t+k'T)=\overline{g_1}(s)g_2(t)
\end{equation}
for all $s,t\in\R$ and all $k,k'\in\Z$ s.t. $kS= ps_1+q s_2$ and $k'T= pt_1+q t_2$ for some $p,q\in\Z.$ Taking (\ref{g1 g2 pm periodic}) into account, the only possibilities are then the two cases in the statement of the proposition; finally, in the second case, $\Gamma$ necessarily satisfies (\ref{condition Gamma}): $\Gamma$ is exactly the set of elements of the form $(kS,k'T)$ where $kS=ps_1+qs_2$ and $k'T=pt_1+qt_2,$ and (\ref{relation periodicity g1g2}), together with (\ref{anti periodicity}), implies that $k\equiv k'\ [2].$
\end{proof}

\subsection{Spinor representation of the flat tori in $S^3$}

We obtain here an explicit description of the flat tori in $S^3;$ it follows from the general spinor representation formula of Theorem \ref{th main result}, written in frames adapted to the tori:  
\begin{thm}\label{th formula F g}
Let $\Gamma$ be a subgroup of translations of $\mathcal{A}\simeq\R^{2}$ satisfying the conditions  (\ref{structure gamma 1}) and (\ref{structure gamma 2}), and $g:\mathcal{A}\rightarrow S^3_{\mathcal{A}}$ be a conformal map such that
\begin{equation}\label{g psi th4}
g'\ g^{-1}=\cos\psi\ J+\sin\psi\ K
\end{equation}
for some conformal map $\psi:\mathcal{A}\rightarrow\mathcal{A},$ satisfying one of the two conditions in Proposition \ref{prop structure g}. Writing 
\begin{equation}\label{psi theta th4}
\psi=\theta_1+\sigma\theta_2,\hspace{1cm} \theta_1,\theta_2:\mathcal{A}\rightarrow\R,
\end{equation}
we moreover assume that $\theta_2$ belongs to $(\pi/2,\pi)$ mod. $\pi.$ Then the formula 
\begin{equation}\label{formula F g}
F=\sigma\ \overline{g}\ \tch{g}\hspace{1cm}\in\hspace{.3cm} S^3\subset\R^4\subset\HA
\end{equation}
defines a flat torus immersed in $S^3.$ Conversely, a flat torus immersed in $S^3$ is of that form.
\end{thm}
\begin{rem}\label{rmk theta2 gamma periodic}
By (\ref{g psi th4}), the functions $\cos\psi$ and $\sin\psi$ appear to be $\Gamma$-periodic, and the function $\psi$ induces a map $\mathcal{A}/\Gamma\rightarrow \mathcal{A}/\Gamma_{2\pi}$ where $\Gamma_{2\pi}=2\pi\Z\oplus\sigma\ 2\pi\Z\subset\mathcal{A}.$ Moreover, since $\theta_2:\mathcal{A}\rightarrow\R$ induces a map $\mathcal{A}/\Gamma\rightarrow \R/2\pi\Z$ and is assumed to belong to $(\pi/2,\pi)$ mod. $\pi,$ $\theta_2$ is in fact $\Gamma$-periodic.
\end{rem}
\begin{proof}
We first assume that $F:M\rightarrow S^3$ is the immersion of a flat torus, and we fix an orientation on $M.$ Since its Gauss map $G:M\rightarrow\mathcal{Q}\subset\Lambda^2\R^4$ is regular (as it is for every surface in $S^3$), we may consider the Lorentz structure induced on $M$ by $G$ and compatible with the orientation; we have $M\simeq \mathcal{A}/\Gamma,$ where $\Gamma$ is a subgroup of translations of $\mathcal{A}\simeq\R^{2}$ which satisfies (\ref{structure gamma 1}) and (\ref{structure gamma 2}). We now consider $\varphi\in\Gamma(\Sigma)$ the restriction to $M$ of the constant spinor field $\unitquat\in\HA$ of $\R^4.$ As above, we consider the pull-backs of the bundles and sections to the universal covering $\pi:\tilde{M}\simeq\mathcal{A}\rightarrow M\simeq \mathcal{A}/\Gamma.$ We fix $\tilde{s},$ a global section of $\pi^*\tilde{Q}$ which is a lift of a parallel frame $(e_0,e_1,e_2,e_3),$ where $e_0$ is here the unit outer normal vector of the sphere and $(e_0,e_1)$ and $(e_2,e_3)$ are respectively normal and tangent to the torus; we have
$$F\simeq\langle\langle e_0\cdot\varphi,\varphi\rangle\rangle,$$
since, for a surface in $S^3,$ the position vector $F$ coincides with the unit outer normal vector of the sphere (see (\ref{ident trivial})), and thus 
$$F\simeq\overline{[\varphi]}[e_0]\tch{[\varphi]}=\sigma\ \overline{g}\ \tch{g}$$
where $g:\tilde{M}\simeq\mathcal{A}\rightarrow S^3_{\mathcal{A}}\subset\HA$ represents the spinor field $\varphi$ in $\tilde{s}.$ Easy computations using (\ref{g psi th4}) and (\ref{psi theta th4}) then yield
\begin{equation}\label{formule dxF}
\partial_xF\ =\ 2\sin\theta_2\ g^{-1}\left(\sin\theta_1J-\cos\theta_1K\right)\tch{g}
\end{equation}
and
\begin{equation}\label{formule dyF}
\partial_yF\ =\ -2\cos\theta_2\ g^{-1}\left(\cos\theta_1J+\sin\theta_1K\right)\tch{g};
\end{equation}
using these formulas, the basis $(\partial_xF,\partial_yF)$ has the orientation of the basis $(dF(e_2),$ $dF(e_3))$ if and only if the determinant 
$$-4\sin\theta_2\cos\theta_2\left|\begin{array}{rr}\sin\theta_1&\cos\theta_1\\-\cos\theta_1&\sin\theta_1\end{array}\right|=-4\sin\theta_2\cos\theta_2$$
is positive ($dF(e_2)\simeq\langle\langle e_2\cdot\varphi,\varphi\rangle\rangle $ is $g^{-1}J\tch{g}$ and $dF(e_3)\simeq\langle\langle e_3\cdot\varphi,\varphi\rangle\rangle$ is $g^{-1}K\tch{g}$ (see (\ref{ident trivial}))), which reads
\begin{equation}\label{theta_2 pi 1}
\theta_2\in(\pi/2,\pi)\ mod\ \pi.
\end{equation}
We now prove the direct statement: we assume that $F$ is given by (\ref{formula F g}); using (\ref{formule dxF}) and (\ref{formule dyF}) we get 
\begin{equation}\label{formulas partial F square}
\overline{ \partial_xF}\ \partial_xF=4\sin^2\theta_2,\hspace{1cm}\overline{ \partial_yF} \ \partial_yF=4\cos^2\theta_2
\end{equation}
and
\begin{equation}\label{formulas partial F cross}
\frac{1}{2}\left(\overline{ \partial_xF}\ \partial_yF+\overline{ \partial_yF}\ \partial_xF\right)=0.
\end{equation}
This implies that $F$ is an immersion if $\theta_2\neq 0\ [\pi/2]$ and that 
$$\frac{1}{2\sin\theta_2}\partial_xF\hspace{.5cm}\mbox{and}\hspace{.5cm}-\frac{1}{2\cos\theta_2}\partial_yF$$ 
form an orthonormal basis tangent to the immersion in that case, which is moreover positively oriented if $\theta_2$ belongs to $(\pi/2,\pi)\mod \pi$ (on the torus, we choose the orientation induced by $F,$ i.e. such that $(\partial_xF,\partial_yF)$ is positively oriented). Thus the Gauss map is given by
$$\tilde{G}=\frac{1}{2\sin\theta_2}\partial_xF\wedge-\frac{1}{2\cos\theta_2}\partial_yF\simeq-\frac{1}{4\sin\theta_2\cos\theta_2}\ \partial_xF\ \tch{\partial_yF}$$
(see (\ref{ident clifford})), which, by (\ref{formule dxF}) and  (\ref{formule dyF}), easily gives $\tilde{G}=g^{-1}Ig.$ Since $g$ is a conformal map, so is $\tilde{G},$ which implies that the immersion is flat with flat normal bundle (see Proposition \ref{prop pull back}). 
\end{proof}
\begin{rem} \label{metric torus S3}
By (\ref{formulas partial F square})-(\ref{formulas partial F cross}), the metric of a flat torus in $S^3$ is given by
$$4(\sin^2\theta_2\ dx^2 +\cos^2\theta_2\ dy^2)$$
in the special chart $a=x+\sigma y.$
\end{rem}
\subsection{Consequence: the description of the flat tori in $S^3$}\label{csq flat tori}

We consider here the Hopf fibration 
\begin{eqnarray}
S^3\hspace{.3cm}\subset\ \HH&\rightarrow& S^2\hspace{.3cm}\subset\ \Im m\ \HH\label{hopf fibration classic}\\
q&\mapsto& q^{-1}Iq;\nonumber
\end{eqnarray}
we recall that a unit speed curve $\gamma:I\rightarrow S^3$ is said to be  horizontal if \
$$\gamma'\gamma^{-1}=\cos\psi\ J+\sin\psi\ K$$ 
for some function $\psi:I\rightarrow\R.$ As a corollary of Theorem \ref{th formula F g}, we obtain the following description of the flat tori in $S^3$ \cite{B,Sa,S,K}:
\begin{thm}\label{thm bianchi} 
A flat torus immersed in $S^3$ may be written as a product
\begin{equation}\label{F g1 g2}
F=\overline{g_1}\ g_2
\end{equation}
in the quaternions, where $g_1$ and $g_2$ are two horizontal closed curves in $S^3\subset\HH.$  
\end{thm}
\begin{proof}
By Theorem \ref{th formula F g}, the torus is represented by a map $g:\mathcal{A}\rightarrow S^3_{\mathcal{A}}$. Writing
$$g\ =\ \frac{1+\sigma}{2}\ g_1\ +\ \frac{1-\sigma}{2}\ g_2$$
with $g_1,g_2\in S^3\subset\HH,$ formula (\ref{formula F g}) gives 
$$F=\frac{\sigma}{2}\left(\overline{g_1}g_2+\overline{g_2}g_1\right)+\frac{1}{2}\left(\overline{g_1}g_2-\overline{g_2}g_1\right)\simeq \overline{g_1}g_2,$$
where we use the identification (\ref{ident quat}). Since $g$ is a conformal and horizontal map, the maps $g_1$ and $g_2$ are curves
$$\begin{array}[t]{rcl}
g_1:\hspace{.3cm}\R&\rightarrow&S^3\\
s&\mapsto &g_1(s)
\end{array}
\hspace{.5cm}\mbox{and}\hspace{.5cm}
\begin{array}[t]{rcl}
g_2:\hspace{.3cm}\R&\rightarrow&S^3\\
t&\mapsto &g_2(t)
\end{array}$$
which are horizontal with respect to the Hopf fibration (\ref{hopf fibration classic}). They are respectively periodic with period $S$ and $T,$ or $2S$ and $2T.$
\end{proof}
\begin{rem}
The curves $g_1$ and $g_2$ in (\ref{F g1 g2}) appear to be the components of the constant spinor field $\varphi=\unitquat_{|M}\ \in\HA$ of $\R^4$ in a frame adapted to the torus in $S^3.$
\end{rem}
\begin{rem}
In the Kitagawa representation \cite{K}, the curves $g_1$ and $g_2$ are constructed as asymptotic curves in $S^3:$ here, they appear as horizontal curves. See Section \ref{section Kitagawa representation} for the precise relation to the Kitagawa representation.
\end{rem}
\begin{rem}
The converse statement of the theorem is true: (\ref{F g1 g2}) defines a flat torus in $S^3$ if it is an immersion, which is guaranteed if $\theta_2\neq 0\ [\pi/2]$ in the statement of Theorem \ref{th formula F g}; assuming that $\theta_2\in(\pi/2,\pi)$ mod. $\pi,$ and since $\theta_2=(\psi_1-\psi_2)/2$ (recall the definition (\ref{def psi1 psi2}) of $\psi_1$ and $\psi_2$), we get equivalently  
$$\psi_1(s)-\psi_2(t)\hspace{.3cm}\in\hspace{.5cm}(\pi,2\pi)\ mod.\ 2\pi$$
for all $s,t\in\R,$ that is, the functions $\psi_1$ and $\psi_2$ take their values in intervals of length $<\pi$ such that
\begin{equation}\label{condition psi 1 2}
(2k+1)\pi<\min\psi_1-\max\psi_2\leq\max\psi_1-\min\psi_2<2(k+1)\pi
\end{equation}
for some $k\in\Z.$ The Kitagawa representation gives a nice interpretation of this condition; see Section \ref{section Kitagawa representation} below.
\end{rem}
\begin{rem}
It may be proved that a horizontal curve $\gamma$ in $S^3,$ parameterized by arc length and biregular, has constant torsion $=\pm 1,$ where the sign depends on the orientation of $S^3.$ Moreover, the curve $\overline{\gamma}$ has the opposite torsion. Thus, if $g_1,$ $g_2$ are biregular curves, they have constant torsion $\pm 1,$ and the surface $F=\overline{g_1}g_2$ is the product of a curve with torsion $1$ by a curve with torsion $-1$; this is the form of the flat tori in $S^3$ constructed in \cite{B,Sa,S}.
\end{rem}

\subsection{The Kitagawa representation of the flat tori in $S^3$}\label{section Kitagawa representation}

We now explain how the previous results are related to the Kitagawa representation of the flat tori in $S^3$ \cite{K}. We assume that $g:\mathcal{A}\rightarrow S^3_{\mathcal{A}}$ satisfies the hypothesis of Theorem \ref{th formula F g}: formula (\ref{formula F g}) (or equivalently (\ref{F g1 g2})) defines a flat torus immersed in $S^3.$ We moreover consider $\psi: \mathcal{A}\rightarrow\mathcal{A}$ such that (\ref{g psi th4}) holds, and we set $\psi_1,\psi_2\in\R$ such that
$$\psi\ =\ \frac{1+\sigma}{2}\ \psi_1\ +\ \frac{1-\sigma}{2}\ \psi_2.$$
We begin with a simple lemma:
\begin{lem}
There exists $\alpha\in\R$ such that 
$$\sin(\psi_1(s)-\alpha)>0\hspace{1cm}\mbox{and}\hspace{1cm}\sin(\psi_2(t)-\alpha)>0$$ 
for all $s,t\in\R.$ In particular we have
$$\alpha\hspace{.3cm}\notin\hspace{.5cm} \psi_1(\R)\cup \psi_2(\R)\hspace{.3cm} \mbox{mod.}\ \pi.$$
\end{lem}
\begin{proof}
From (\ref{condition psi 1 2}) we get
$$\max\psi_2+(2k+1)\pi<\min\psi_1\hspace{.5cm}\mbox{and}\hspace{.5cm}\max\psi_1<\min\psi_2+2(k+1)\pi.$$
We take $\alpha$ such that
$$\max\psi_2+(2k+1)\pi<\alpha<\min\psi_1.$$
We have
$$0<\psi_1-\alpha<\max\psi_1-(\max\psi_2+(2k+1)\pi),$$
and since
$$\max\psi_1-(\max\psi_2+(2k+1)\pi)<(\min\psi_2+2(k+1)\pi)-(\max\psi_2+(2k+1)\pi)\leq \pi,$$
we get $\sin(\psi_1-\alpha)>0.$ Similarly, we have
$$\psi_2-\alpha\leq\max\psi_2-\alpha<\max\psi_2-(\max\psi_2+(2k+1)\pi)=-(2k+1)\pi$$
and 
$$\psi_2-\alpha\geq\min\psi_2-\alpha>\min\psi_2-\min\psi_1\geq\min\psi_2-\max\psi_1>-2(k+1)\pi,$$
and thus $\sin(\psi_2-\alpha)>0.$ 
\end{proof}
For such a number $\alpha,$ we set
$$J_{\alpha}\ =\ \cos\alpha\ J\ +\ \sin\alpha\ K\hspace{.3cm}\in\ \HH$$
and consider the Hopf fibration
\begin{eqnarray*}
h_\alpha:\hspace{1cm}S^3\hspace{.3cm}\subset\ \HH&\rightarrow& S^2\hspace{.3cm}\subset\ \Im m\ \HH\\
u&\mapsto& u^{-1}\ J_{\alpha}\ u.
\end{eqnarray*}
We then consider the curves
$$\gamma_1=h_{\alpha}(g_1)\hspace{.5cm}\mbox{and}\hspace{.5cm}\gamma_2= h_\alpha(g_2),$$
where $g_1$ and $g_2$ are such that
$$g\ =\ \frac{1+\sigma}{2}\ g_1\ +\ \frac{1-\sigma}{2}\ g_2.$$
\begin{lem}
The curves $\gamma_1,\gamma_2:S^1\rightarrow S^2$ are immersions with geodesic curvatures
\begin{equation}\label{geod curv gamma}
k_{\gamma_1}=cotan (\psi_1-\alpha)\hspace{.5cm}\mbox{and}\hspace{.5cm}k_{\gamma_2}=cotan(\psi_2-\alpha).
\end{equation}
They moreover satisfy
\begin{equation}\label{geod curv vide}
k_{\gamma_1}(S^1)\ \cap\ k_{\gamma_2}(S^1)\ =\ \emptyset.
\end{equation}
We assume here that $S^2$ is oriented w.r.t. its inner unit normal vector.
\end{lem}
\begin{proof}
We have 
\begin{eqnarray*}
{\gamma'}_1&=&\overline{g_1'}J_{\alpha}g_1+\overline{g_1}J_\alpha g_1'\\
&=&-\overline{g_1}(\cos\psi_1 J+\sin\psi_1K)J_{\alpha}g_1+\overline{g_1}J_\alpha(\cos\psi_1 J+\sin\psi_1K)g_1\\
&=&2\sin(\psi_1-\alpha)\ \overline{g_1}Ig_1,
\end{eqnarray*}
which is not 0, since $\alpha\notin\psi_1(\R)\ [\pi].$ Thus (and since $\sin(\psi_1-\alpha)>0$)
$$\frac{\gamma'_1}{|\gamma'_1|}=\overline{g_1}Ig_1.$$
By an analogous computation, we get
$$\left(\frac{\gamma'_1}{|\gamma'_1|}\right)'=\left(\overline{g_1}Ig_1\right)'=2\cos\psi_1\overline{g_1}Kg_1-2\sin\psi_1\overline{g_1}Jg_1;$$
inserting finally
$$J_{\alpha}\ =\ \cos\alpha\ J\ +\ \sin\alpha\ K\hspace{.5cm}\mbox{and}\hspace{.5cm}K_{\alpha}\ =-\ \sin\alpha\ J\ +\ \cos\alpha\ K$$
in the last formula, we get
\begin{equation}\label{geod curve gamma1}
\left(\frac{\gamma'_1}{|\gamma'_1|}\right)'=2\cos(\psi_1-\alpha)\overline{g_1}K_{\alpha}g_1-2\sin(\psi_1-\alpha)\overline{g_1}J_{\alpha}g_1.
\end{equation}
The last term is normal to $S^2$ at the point $\gamma_1=\overline{g_1}J_{\alpha}g_1,$ and the tangent frame $\overline{g_1}Ig_1,$ $\overline{g_1}K_{\alpha}g_1$ is positively oriented if $S^2$ is oriented by the inner normal vector $-\overline{g_1}J_{\alpha}g_1.$ The geodesic curvature of $\gamma_1$ in $S^3$ is thus given by the coefficient of the first term, divided by $2\sin(\psi_1-\alpha)$ (the derivative in (\ref{geod curve gamma1}) has to be taken with respect to arc length), which is the first formula in (\ref{geod curv gamma}). Finally, (\ref{geod curv vide}) is a consequence of the two formulas in (\ref{geod curv gamma}) and of the condition
$$2\theta_2=\psi_1-\psi_2\hspace{.3cm} \in\hspace{.5cm}(\pi,2\pi)\ [2\pi]$$
together with the property
$$cotan(a)=cotan(b)\hspace{.5cm}\mbox{iff}\hspace{.5cm}a-b=0\ [\pi]$$
for all $a,b\neq 0\ [\pi].$
\end{proof}
We now consider the double covering
\begin{eqnarray*}
p_\alpha:\hspace{1cm}S^3&\rightarrow& US^2\\
u&\mapsto& (u^{-1}\ J_{\alpha}\ u\ ,\ u^{-1}\ I\ u),
\end{eqnarray*}
where $US^2$ denotes the unit tangent bundle of $S^2.$ If $c:\R\rightarrow S^2$ is an immersion, a curve $\hat{c}:\R\rightarrow S^3$ such that $p\circ \hat{c}=(c,c'/|c'|)$ is said to be an asymptotic lift of $c,$ which means that the curve $\hat{c}$ belongs to the Hopf cylinder $h_{\alpha}^{-1}(c)\subset S^3,$ and is such that $B(\hat{c}',\hat{c}')=0,$ where $B$ is the second fundamental form of $h_{\alpha}^{-1}(c)$ in $S^3;$ we refer to \cite{K} for more information concerning asymptotic lifts.
\\

We finally note that the curves $g_1$ and $g_2$ are asymptotic lifts of $\gamma_1$ and $\gamma_2:$ 
\begin{lem}
We have
$$p_{\alpha}(g_1)=\left(\gamma_1,\frac{\gamma'_1}{|\gamma'_1|}\right)\hspace{.5cm}\mbox{and}\hspace{.5cm}p_{\alpha}(g_2)=\left(\gamma_2,\frac{\gamma'_2}{|\gamma'_2|}\right).$$
\end{lem}
\begin{proof}
We already noticed that
$$\gamma_1=\overline{g_1}J_{\alpha}g_1\hspace{1cm}\mbox{and}\hspace{1cm}\frac{\gamma'_1}{|\gamma'_1|}=\overline{g_1}Ig_1.$$
\end{proof}
The results in Theorems \ref{th formula F g} and \ref{thm bianchi} may thus be interpreted as follows:  a flat torus immersed in $S^3$ is a product of the form (\ref{F g1 g2}) in $S^3,$ where $g_1$ and $g_2$ are asymptotic lifts  of two curves $\gamma_1,\gamma_2$ satisfying (\ref{geod curv vide}). This is the Kitagawa representation of the flat tori in $S^3;$ see \cite{K}, and also \cite{GM,W}. 

\subsection{The Gauss map image of a flat torus in $S^3$}

Since the map $g$ appears to be a lift of the Gauss map (Proposition \ref{prop diag}), we easily deduce the structure of the Gauss map image of the flat tori in $S^3$ \cite{E,W}:
\begin{cor}\label{corollary gauss map}
Let us consider the Gauss map 
$$G:\mathcal{A}/\Gamma\rightarrow S^2_{\mathcal{A}}\simeq S^2\times S^2$$
of a flat torus $F:\mathcal{A}/\Gamma\hookrightarrow S^3.$ Its image is a product of closed curves $\gamma_1\times\gamma_2$ whose geodesic curvatures $k_
{\gamma_1}$ and  $k_{\gamma_2}$ satisfy
\begin{equation}\label{ineq geod curv}
-\pi<\int_{I}k_{\gamma_1}-\int_{J}k_{\gamma_2}<\pi
\end{equation}
for all subintervals $I,J$ of $\R,$ and
\begin{equation}\label{total curvature zero}
\int_{\gamma_1}k_{\gamma_1}=\int_{\gamma_2}k_{\gamma_2}=0.
\end{equation}
\end{cor}
\begin{proof} 
By Theorem \ref{th formula F g}, $F=\sigma \overline{g}\tch{g},$ where $g:\mathcal{A}\rightarrow S^3_{\mathcal{A}}$ is a horizontal and conformal map; moreover, the map $g$ is of the form
$$g=(g_1,g_2)$$
where $s\mapsto g_1(s)$ and $t\mapsto g_2(t)$ are two closed curves, respectively with period $S$ and $T,$ or $2S$ and $2T$ (Proposition \ref{prop structure g}). Since $g$ is a lift of the Gauss map $G$, the image of $G$ is also a product of closed curves $\gamma_1\times\gamma_2$ (precisely, $\gamma_i$ is the projection of $g_i$ by the Hopf fibration (\ref{hopf fibration classic}), $i=1,2$). Moreover, writing
\begin{equation}\label{int psi}
\psi(a_2)-\psi(a_1)=\int_{(a_1,a_2)}d\psi
\end{equation}
for all $a_1,a_2\in\mathcal{A},$ and $\psi(s,t)=\frac{1+\sigma}{2}\ \psi_1(s)+\frac{1-\sigma}{2}\ \psi_2(t),$ we get
\begin{eqnarray*}
d\psi&=&\partial_s\psi\ ds+\partial_t\psi\ dt\\
&=&\frac{1}{2}\left(\psi'_1\ ds+\psi'_2\ dt\right)+\frac{\sigma}{2}\left(\psi'_1\ ds-\psi'_2\ dt\right),
\end{eqnarray*}
and thus, taking the $\sigma$- component of (\ref{int psi}),
\begin{eqnarray}
\theta_2(a_2)-\theta_2(a_1)&=&\frac{1}{2}\left(\int_{(s_1,s_2)}\psi'_1\ ds-\int_{(t_1,t_2)}\psi'_2\ dt\right)\nonumber\\
&=&\frac{1}{2}\left(\int_{I}k_{\gamma_1}-\int_{J}k_{\gamma_2}\right)\label{theta2 geodesic curvatures},
\end{eqnarray}
where $\gamma_1:I\rightarrow S^2$ (resp. $\gamma_2:J\rightarrow S^2$) is the projection of $g_1:(s_1,s_2)\rightarrow S^3$ (resp. $g_2:(t_1,t_2)\rightarrow S^3$) parameterized by arc length. The last equality is a consequence of the fact that $\psi'_1$ and $\psi'_2$ are the geodesic curvatures of $g_1$ and $g_2$ (Remark \ref{rmk interpretation psi 1 2}) and that the integrals with respect to arc length of the geodesic curvatures of $g_1$ and $g_2$ and of their projections on $S^2$ coincide.

Since $\theta_2$ belongs to $(\pi/2,\pi)$ mod. $\pi,$ $\theta_2(a_2)-\theta_2(a_1)$ belongs to $(-\pi/2,\pi/2)$ and (\ref{ineq geod curv}) follows. Equation (\ref{total curvature zero}) is also a consequence of (\ref{theta2 geodesic curvatures}) together with the fact that $\theta_2$ is $\Gamma$-periodic (Remark \ref{rmk theta2 gamma periodic}).
\end{proof}

\appendix
\section{Auxiliary results on Lorentz numbers and quaternions}

\subsection{Invertible elements in the algebra $\HA$}
We describe here the invertible elements in the algebra of quaternions with coefficients in $\mathcal{A}.$ We first note that the set of invertible elements of $\mathcal{A}$ is
\begin{equation}\label{invertible in A}
\mathcal{A}^*=\mathcal{A}\backslash(1\pm\sigma)\R.
\end{equation}
\begin{lem}\label{invertible HA}
Let us define
$$\HA_+=\{\xi\in\HA:\ \sigma\xi=\xi\}\hspace{.5cm}\mbox{and}\hspace{.5cm}\HA_-=\{\xi\in\HA:\ \sigma\xi=-\xi\}.$$
The set of invertible elements of $\HA$ is
$${\HA}^{*}=\HA\backslash(\HA_+\cup \HA_-).$$
\end{lem}
\begin{proof}
Let $\xi=a_0\unitquat+a_1I+a_2J+a_3K\in\HA,$ with $a_i=u_i+\sigma v_i,$ $u_i,v_i\in\R.$ A straightforward computation yields
\begin{equation}\label{H xi coord}
H(\xi,\xi)=\overline{\xi}\xi=\sum_{i=0}^{3}(u_i^2+v_i^2)+2\sigma\sum_{i=0}^3u_iv_i.
\end{equation}
We note that $\xi$ is invertible in $\HA$ if and only if $H(\xi,\xi)$ is invertible in $\mathcal{A}$ (the inverse of $\xi$ is then $\overline{\xi}/H(\xi,\xi)$). Thus $\xi$ is not invertible in $\HA$ if and only if
$$\sum_{i=0}^3(u_i^2+v_i^2)=2\varepsilon\sum_{i=0}^3u_iv_i,$$
with $\varepsilon=\pm 1$ (by (\ref{invertible in A})). This gives $\sum_i(u_i-\varepsilon v_i)^2=0,$ and thus $\xi\in \HA_+\cup \HA_-,$ since $\HA_+$ and $\HA_-$ are explicitly given by
$$\HA_+=\{(1+\sigma)q,\ q\in\HH\}\hspace{1cm}\mbox{and}\hspace{1cm}\HA_-=\{(1-\sigma)q,\ q\in\HH\}.$$
\end{proof}

\subsection{Square roots in the Lorentz numbers}

\begin{lem}\label{lem square}
Let $b\in\mathcal{A}.$ There exists $a\in\mathcal{A}$ such that
\begin{equation}\label{square}
a^2=b
\end{equation}
if and only if $b$ belongs to the cone 
$$\mathcal{C}=\{u+\sigma v\in\mathcal{A},\ u\geq 0,\ v\in\R:\ -u\leq v\leq u\};$$ 
moreover, equation (\ref{square}) has exactly four solutions if $b$ belongs to the interior of $\mathcal{C}.$ In particular, if $\xi\in\HA$ is invertible, then equation 
$$a^2=H(\xi,\xi)$$
has four solutions in $\mathcal{A},$ which are moreover invertible in $\mathcal{A}.$ 
\end{lem}

\begin{proof}
Setting $a=x+\sigma y$ and $b=u+\sigma v,$ equation (\ref{square}) reads
$$x^2+y^2=u\hspace{1cm}\mbox{and}\hspace{1cm}2xy=v,$$
which is solvable if and only if $u\geq 0$ and $-u\leq v\leq u.$ If these conditions hold, the solutions are
$$x=\frac{1}{2}\left(\varepsilon_1\sqrt{u+v}+\varepsilon_2\sqrt{u-v}\right)\hspace{.5cm}\mbox{and}\hspace{.5cm}y=\frac{1}{2}\left(\varepsilon_1\sqrt{u+v}-\varepsilon_2\sqrt{u-v}\right),$$
where $\varepsilon_1,\varepsilon_2=\pm 1,$ and the first part of the lemma follows. For the last claim, we note that formula (\ref{H xi coord}) implies that $H(\xi,\xi)$ belongs to the interior of the cone $\mathcal{C}$ whenever it is invertible; the solutions are then moreover clearly invertible in $\mathcal{A}.$
\end{proof}

\end{document}